\documentclass[12pt]{amsart}

\usepackage{amsmath, amscd, amssymb, amsthm,amsfonts}
\usepackage{stmaryrd,euscript, mathrsfs, latexsym}
\usepackage{color}
\usepackage[T1]{fontenc} 

\ifx\pdfoutput\undefined
\usepackage{graphicx}
\else
\fi

\usepackage{tikz}
\usetikzlibrary{matrix,arrows,positioning}
\tikzset{node distance=2cm, auto}

\usepackage[nodayofweek]{datetime}

\parskip=\medskipamount
\DeclareMathOperator{\Ho}{Ho}
\DeclareMathOperator{\tr}{Tr}

\DeclareMathOperator{\Cob}{Cob}
\DeclareMathOperator{\PCob}{Pre-Cob}

\DeclareMathOperator{\Tw}{Tw}

\DeclareMathOperator{\Endo}{End}

\DeclareMathOperator{\Kom}{Kom}
\DeclareMathOperator{\Ann}{Ann}

\DeclareMathOperator{\Mat}{Mat}

\DeclareMathOperator{\Morph}{Hom}

\DeclareMathOperator{\Cone}{Cone}

\DeclareMathOperator{\Univ}{U}
\DeclareMathOperator{\TL}{TL}

\DeclareMathOperator{\K}{K}

\DeclareMathOperator{\Tot}{Tot}
\DeclareMathOperator{\Hom}{Hom}

\newcommand{\conj}[1]{\quad\textnormal{ #1 }\quad}

\newcommand{\inp}[1]{\ensuremath{\langle #1 \rangle}}

\newcommand{\coneqnOne}[0]{(8.1)}

\newcommand{\normaltext}[1]{\textnormal{#1}}

\newcommand{\quantsl}[0]{\Univ_q\mathfrak{sl}(2)}

\newcommand{\quantsu}[0]{\Univ_q\mathfrak{sl}(2)}

\makeatletter
\def\imod#1{\allowbreak\mkern2.5mu({\operator@font mod}\,#1)}
\makeatother

\renewcommand{\a}{\alpha}

\newcommand{\e}{\epsilon}

\newcommand{\opp}{\oplus}
\newcommand{\ott}{\otimes}

\renewcommand{\d}{\delta}

\newcommand{\CPic}[1]{
\begin{minipage}{.45in}
\includegraphics[scale=.75]{#1}
\end{minipage}
}

\newcommand{\CPPic}[1]{
\begin{minipage}{.8in}
\includegraphics[scale=1.1]{#1}
\end{minipage}
}

\newcommand{\MPic}[1]{
\begin{minipage}{.35in}
\includegraphics[scale=.45]{#1}
\end{minipage}
}

\newcommand{\oPic}[1]{
\begin{minipage}{.2in}
\includegraphics[scale=.3]{#1}
\end{minipage}
}

\newcommand{\aA}{\mathcal{A}}

\newcommand{\aC}{\mathcal{C}}

\newcommand{\aE}{\mathcal{E}}

\newcommand{\aL}{\mathcal{L}}

\newcommand{\aO}{\mathcal{O}}

\newcommand{\aQ}{\mathcal{Q}}

\newcommand{\aT}{\mathcal{T}}

\usepackage{bbm}

\newcommand{\CC}{\mathbb{C}}

\newcommand{\ZZ}{\mathbb{Z}}

\newcommand{\eA}{\EuScript{A}}

\newcommand{\eZ}{\EuScript{Z}}

\theoremstyle{plain}

\newtheorem{theorem}[subsection]{Theorem}
\newtheorem*{statement}{Statement}

\newtheorem{proposition}[subsection]{Proposition}

\newtheorem{corollary}[subsection]{Corollary}
\newtheorem{lemma}[subsection]{Lemma}
\theoremstyle{remark}

\newtheorem{remark}[subsection]{Remark}
\newtheorem{vista}[subsection]{Vista}
\newtheorem{notation}[subsection]{Notation}
\newtheorem{observation}[subsection]{Observation}
\theoremstyle{definition}

\newtheorem{example}[subsection]{Example}

\newtheorem{definition}[subsection]{Definition}

\numberwithin{equation}{section}

\usepackage{lineno}
\usepackage{tikz-cd}

\makeatletter
\def\imod#1{\allowbreak\mkern2.5mu({\operator@font mod}\,#1)}
\makeatother

\newcommand{\dominates}{\unrhd}

\newcommand{\dominatedBy}{\unlhd}
\newcommand{\ndominatedBy}{\not{\hskip-7pt}\unlhd}

\DeclareMathOperator{\Id}{Id}
\DeclareMathOperator{\op}{op}

\renewcommand{\circle}{\oPic{smallCircle}\hskip2pt}

\newcommand{\CPicTailbox}[1]{\Bigg[\CPic{#1}\hspace{-.04in}\Bigg]}
\newcommand{\CPicTailboxx}[1]{\Bigg[\CPic{#1}\hspace{.06in}\Bigg]}
\newcommand{\CPicTailboxxTwo}[1]{\Bigg[\Bigg[\CPic{#1}\hspace{.06in}\Bigg]\Bigg]}
\newcommand{\MPicTailboxx}[1]{\Big[\MPic{#1}\hspace{-.03in}\Big]}

\newcommand{\xto}[1]{\xrightarrow{#1}}
\newcommand{\mf}[1]{\mathfrak{#1}}

\begin{document}
\title[An Exceptional Collection for Khovanov Homology]{An Exceptional Collection for Khovanov Homology}

\author[Benjamin Cooper and Matt Hogancamp]{Benjamin Cooper and Matt Hogancamp}
\address{Institut f\"{u}r Mathematik, Universit\"{a}t Z\"{u}rich, Winterthurerstrasse 190, CH-8057 Z\"{u}rich}
\email{benjamin.cooper\char 64 math.uzh.ch}
\address{Department of Mathematics, University of Virginia, Charlottesville, VA 22904}
\email{mhoganca\char 64 gmail.com}

\begin{abstract}
  The Temperley-Lieb algebra is a fundamental component of $SU(2)$
  topological quantum field theories.  We construct chain complexes
  corresponding to minimal idempotents in the Temperley-Lieb algebra. Our
  results apply to the framework which determines Khovanov
  homology. Consequences of our work include semi-orthogonal decompositions
  of categorifications of Temperley-Lieb algebras and Postnikov
  decompositions of all Khovanov tangle invariants.
\end{abstract}

\maketitle

\tableofcontents


\section{Introduction}
The purpose of this paper is to give a general method for lifting an
idempotent decomposition of the Temperley-Lieb algebra $\TL_n$ to a
decomposition of its categorification. Roughly speaking, we lift each
idempotent $p\in \TL_n$ to a chain complex $P \in\Kom(n)$ so that the equation
$$p\cdot p = p \conj{ lifts to } P\otimes P \simeq P.$$

Along the way, we find new skein theoretic expressions for the decomposition
of the Temperley-Lieb algebra, the categories which determine Khovanov
homology are extended to differential graded categories, some mapping spaces
are computed and a complete decomposition of these categories is introduced.

\subsection{The Temperley-Lieb algebra and its decomposition}
A $k$-algebra $A$ can be often be expressed as a sum:
\begin{equation}\label{topeq}
A = A_1 \opp A_2 \opp \cdots \opp A_N.
\end{equation}
This information can be encoded by a collection of elements: $\{p_i\}_{i=1}^N$ called {\em projectors} or {\em idempotents}. Each $p_i\in A$ determines a projection and an inclusion
$$A \xto{p_i\cdot} A_i \hookrightarrow A$$
to and from the subspace $A_i$ in $A$. Equation \eqref{topeq} implies that the collection $\{p_i\}$ satisfies the equations:
\begin{equation}\label{idempeqn}
p_i \cdot p_i = p_i, \conj{ } p_i \cdot p_j = 0  \normaltext{ when } i\ne j\conj{ and } 1_A = \sum_{i=1}^N p_i.
\end{equation}
If each projector $p_i$ cannot be written as a sum of two non-trivial projectors then the collection
$\{p_i\}$ is a {\em complete set of primitive mutually-orthogonal
  idempotents}.

Suppose that $\quantsl$ is the quantum group associated to the Lie algebra
$\mf{sl}(2)$ and $V = V_1$ is the 2-dimensional irreducible representation
corresponding to the action of $\mf{sl}(2)$ on $\CC^2$. The {\em Temperley-Lieb
algebra} $\TL_n$ is the endomorphism algebra of the $n$-fold tensor power of $V$:
$$\TL_n = \Endo_{\quantsl}(V^{\ott n}).$$
The inner product $V\ott V \to \CC(q)$ and its dual $\CC(q) \to V\ott V$
generate the Temperley-Lieb algebra $\TL_n$. Drawing the latter as a cup and
the former as a cap gives rise to a pictorial representation of every
element in $\TL_n$. For example, the composition
$$V\ott V \to \CC(q) \to V\ott V\conj{ is pictured as } \MPic{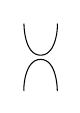}\in \TL_2.$$
This graphical interpretation leads to an alternative, topological, definition of the
Temperley-Lieb algebras as quotients of categories of 1-dimensional
cobordisms. An important consequence of these two different definitions is
the ability to characterize constructions involving the Temperley-Lieb
algebra in two very different ways.

The Temperley-Lieb algebra is semi-simple and a complete set of primitive
mutually-orthogonal idempotents can be described using representation
theory. If $V_k$ is the $k$th irreducible representation of $\mf{sl}(2)$
then the tensor product $V_1^{\ott n}$ can be written as a direct sum of
irreducible representations:
\begin{equation}\label{abveq}
V_1^{\otimes n} \cong V_n \oplus m_{n-2} V_{n-2} \oplus m_{n-4} V_{n-4} \oplus \cdots.
\end{equation}
The number $m_{k} = \dim_{\CC(q)} \Hom(V_k, V_1^{\ott n})$ is the {\em multiplicity} of $V_k$ in $V_1^{\ott n}$.
Corresponding to each summand $V_k \subset m_k V_k \subset V_1^{\ott n}$ is an idempotent
$$p : V_1^{\ott n} \to V_k \hookrightarrow V_1^{\ott n} \in \TL_n.$$
For example, the multiplicity $m_n$ of $V_n$ in $V_1^{\ott n}$ is always
equal to $1$ and the idempotent associated to $V_n \subset V_1^{\ott n}$ is
the famous {\em Jones-Wenzl projector} \cite{WN}. In this paper, we are
interested in the entire collection:
$$\{p_W : W \subset V_1^{\ott n} \normaltext{ such that } W\cong V_k \normaltext{ for some } k \normaltext{ in } \eqref{abveq} \}.$$
Unfortunately, the definition given above is not useful in practice. In
Section \ref{TL section}, we will begin by finding more convenient
expressions for these idempotents in terms of the Jones-Wenzl projectors.

\begin{example}\label{introex1}
When $n=4$, there is a decomposition:
$$V_1^{\ott 4} = V_4 \opp 3 V_2 \opp 2 V_0,$$
which contains a unique copy of the $5$-dimensional irreducible
representation $V_4$, three copies of the $3$-dimensional representation $V_2$
and two copies of the $1$-dimensional trivial representation $V_0$. There are six
idempotents $p_{\e}$ in the Temperley-Lieb algebra $\TL_4$; each
corresponding to projecting onto distinct irreducible summands.  They
are pictured below:
$$\begin{array}{lll}
p_{4,4} & = p_{(1,1,1,1)} & = \MPic{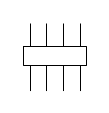}\\
p_{4,2} & = p_{(1,1,1,-1)} + p_{(1,-1,1,1)} + p_{(1,1,-1,1)} & = \frac{[3]}{[4]} \MPic{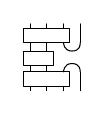} + \frac{1}{[2]} \MPic{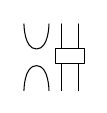} + \frac{[2]}{[3]} \MPic{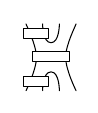}\\
p_{4,0} & = p_{(1,-1,1,-1)} + p_{(1,1,-1,-1)} & = \frac{1}{[2]^2} \MPic{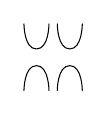} + \frac{1}{[3]} \MPic{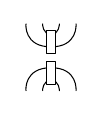}.\\
\end{array}$$
In this illustration, the number $[n] \in\ZZ[q,q^{-1}]$ is the $n$th quantum integer:
$$[n] = q^{-(n-1)} + q^{-(n-3)} + \cdots +  q^{n-3} + q^{n-1}.$$
The subscript $\e = (\e_1,\ldots,\e_4)$ indicates that $p_\e$ corresponds to projection onto a distinct summand isomorphic to $V_{\e_1+\cdots+\e_4}$, see Definition \ref{domOrderstuff}. The $k$th row corresponds to the $k$th isotypic or {\em higher order} projector $p_{n,k}\in\TL_n$. The boxes in the pictures above represent Jones-Wenzl projectors $p_n\in \TL_n$.
\end{example}

\subsection{Categorifications and decompositions}

The idea behind categorification is to replace a $k$-algebra $A$ by a
monoidal category $\eA$. When $\eA$ is monoidal the Grothendieck group
$\K_0(\eA)$ becomes a ring. A category $\eA$ {\em categorifies} a $k$-algebra $A$ when
there is an isomorphism:
\begin{equation}\label{catisoeq}
\K_0(\eA)\ott_{\ZZ} k \xto{\sim} A.
\end{equation}
There are as many examples of categorifications as there are monoidal
categories for which the Grothendieck group functor can be defined. In order
to ensure that something interesting happens one usually asks for the
category $\eA$ to satisfy some additional properties.

In \cite{Kh00}, Khovanov introduced a categorification of the Jones
polynomial. In subsequent papers \cite{KH, DBN}, this homological invariant
of links was refined to a local invariant of tangles, taking values in
categories $\Kom(n)$. There are isomorphisms:
$$\K_0(\Kom(n))\ott_{\ZZ} \CC(q) \xto{\sim} \TL_n,$$
making the categories $\Kom(n)$ categorifications of the Temperley-Lieb
algebras $\TL_n$. In addition to being categorifications, they satisfy the
constraint that they determine knot invariants.

Other categorifications of the Temperley-Lieb algebra have been shown to
lead to knot invariants. Some come from derived categories of coherent
sheaves, others come from enumerative invariants of Lagrangian fibrations,
perverse sheaves on Grassmannians, the category $\aO$, matrix factorizations
or sheaves concentrated on type $A_2$ singularities, \cite{CauK1, SeiS,
  Stroppel, BFK, KhRoz1, Wu, Orlov}. However, the requirement that such
categories determine knot invariants is very strong \cite{KhFrob}. The
categories $\Kom(n)$ are minimal with respect to these constraints. This
makes them the ideal setting for constructions which apply to other
categorifications in this family.

If an algebra $A$ is categorified by $\eA$ then an element $p\in A$ is
categorified by a choice of $P\in \eA$ for which $\K_0(P) = p$. Whenever one
encounters an element $p$ in an algebra $A$, one can ask for lifts $P\in
\eA$. Equation \eqref{catisoeq} implies that there will be at least one lift,
but there may be others depending upon how many extensions exist between
objects in the category $\eA$.

In \cite{CK}, the Jones-Wenzl projector $p_n\in\TL_n$ was lifted to a chain
complex $P_n \in\Kom(n)$ which satisfies $\K_0(P_n) = p_n$. In addition to
being idempotent, $P_n\ott P_n\simeq P_n$, the lift $P_n$ was characterized
uniquely up to homotopy in the category $\Kom(n)$. The properties which
determine $P_n$ are given in terms of the topological description of the
category $\Kom(n)$ and provide a unique counterpoint to algebraic
descriptions that can be obtained by Yoneda's Lemma or localization (as in
Section \ref{postnikovdecompsec}).  There are deep relations between the
categorified projectors $P_n$ and constructions in mathematical physics and
algebraic geometry \cite{GOR1, GS, R3}.

One of the main results of this paper is the construction of chain complexes
$P_\e\in \Kom(n)$ corresponding to each of the idempotents $p_\e \in
\TL_n$ which were discussed in the previous section. More precisely, we show that:
\begin{itemize}
\item {\bf Theorem}  For each $\e$, there is a chain complex $P_\e\in \Kom(n)$ such that $\K_0(P_\e) = p_\e$.

\item {\bf Theorem}  The chain complexes $P_\e$ are idempotent and mututally orthogonal:
$$P_\e \ott P_\e \simeq P_e \conj{ and } P_\e \ott P_\d \simeq 0 \normaltext{ when } \e \ne \d.$$

\item {\bf Theorem}  The projectors $P_\e$ glue together to form a chain complex $R_n$ which satisfies:
$$1_n \simeq R_n$$
\noindent where $1_n\in\Kom(n)$ is the monoidal identity. This homotopy equivalence corresponds to the second part of Equation \eqref{idempeqn} above.
\end{itemize}

An important new ingredient in this setting is the maps between objects. In
the Temperley-Lieb algebra, different subspaces did not interact because
they corresponded to the images of distinct irreducible summands of
$V_1^{\ott n}$. After lifting the idempotents $p_\e$ defining these
subspaces to objects $P_\e \in \Kom(n)$, we find that they {\em must}
interact: there are non-trivial maps between idempotents. However, a more
refined statement can be made. We address the question of what this
interaction looks like in the Theorem below.
\begin{itemize}
\item {\bf Theorem} The mapping spaces between projectors which do not respect the dominance order are contractible:
$$\e\ndominatedBy\d \conj{ implies } \Hom^*(P_\e,P_\d)\simeq 0,$$ 
see Section \ref{homspacessec}.
\end{itemize}

It follows from this theorem that all of the objects in the categories
$\Kom(n)$ are filtered by the projectors $P_\e$. In turn, this filtration
can be used to define Postnikov towers for all objects, including tangle
invariants. The idempotents $\{P_\e\}$ form an {\em exceptional collection}
for Khovanov homology.

\begin{example}\label{introex2}
  The picture of the resolution of identity $R_4$ captures many aspects of
  the information conveyed above. Each projector $p_\e$ in Example
  \ref{introex1} lifts to a chain complex $P_\e$ that satisfies $\K_0(P_\e)
  = p_\e$. These categorified projectors form the vertices of the diagram
  below. The arrows represent non-trivial maps between projectors that
  control the decomposition of the category $\Kom(4)$. There are no arrows
  pointing from right to left.

\begin{center}
\begin{tikzpicture}[scale=10, node distance=2.5cm]
\node (Z1) {$1_4$};
\node (Z2) [right=.25cm of Z1] {$\simeq$};
\node (N2) [right=.25cm of Z2] {$P_{(1,1,-1,-1)}$};
\node (A2) [above of=N2] {$P_{(1,-1,1,-1)}$};
\node (B2) [right of=A2] {$P_{(1,-1,1,1)}$};
\node (C2) [right of=B2] {$P_{(1,1,1,1)}$};
\node (D2) [below of=B2] {$P_{(1,1,-1,1)}$};
\node (E2) [below of=D2] {$P_{(1,1,1,-1)}$};
\draw[->] (A2) to node {} (N2);
\draw[->] (N2) to node {} (D2);
\draw[->] (A2) to node {} (B2);
\draw[->] (B2) to node {} (C2);
\draw[->] (B2) to node {} (D2);
\draw[->] (D2) to node {} (E2);
\draw[->, bend right=45] (B2) to node {} (E2);
\draw[->, bend right] (D2) to node {} (C2);
\draw[->, bend right=50] (E2) to node {} (C2);
\end{tikzpicture}
\end{center}
\end{example}

\section{The Temperley-Lieb category and higher order projectors} \label{TL section}
In this section we summarize basic information about the Temperley-Lieb
category, $\TL$, the Temperley-Lieb algebra, $\TL_n$, and the Jones-Wenzl
projectors, $p_n\in \TL_n$. In Section \ref{idempotentsirred sec}, the
higher order projectors $p_{n,k}\in\TL_n$ are defined representation
theoretically. In Section \ref{genjwprojsec}, new projectors $p_\e\in\TL_n$ are
introduced and related to the higher order projectors. This allows us to
characterize $p_{n,k}$ uniquely in terms of its interaction with other
Temperley-Lieb elements. It is this latter definition which will lift to
chain complexes in Section \ref{charprojsec}. For more information about the
Temperley-Lieb algebra and its connection to low-dimensional topology see
\cite{KL}.

\subsection{Temperley-Lieb category}\label{TLcat}

Here we define the Temperley-Lieb category $\TL$ and establish some basic
notions, such as the through-degree $\tau(a)$ of elements $a\in\TL$.

\begin{definition}\label{TL alg}
The {\em Temperley-Lieb category} $\TL$ is the category of $\quantsu$-equivariant
maps from $n$-fold to $k$-fold tensor powers of the fundamental
representation $V$. 

More specifically, the objects of $\TL$ are indexed by integers $n$
corresponding to tensor powers, $V^{\otimes n}$, of the fundamental
representation and the morphisms
\begin{equation}\label{sl2interp}
 \TL(n,k) = \Morph_{\quantsl}(V^{\otimes n},V^{\otimes k}).
\end{equation}
are determined by $\quantsl$-equivariant maps. The elements of $\TL(n,k)$
can be represented by $\CC(q)$-linear combinations of pictures consisting of
chords from a collection of $n$ points to a collection of $k$ points
situated on two horizontal lines in the plane. Such pictures correspond to
compositions of the maps
$$\CC(q) \to V\ott V \conj{ and } V \ott V\to \CC(q).$$
Pictures are considered equivalent when they are isotopic relative to the
boundary. We also impose the relation that a disjoint circle can be removed
at the cost of multiplying by the graded dimension of $V$: $q + q^{-1}$.
\end{definition}

A sample element of the space of morphisms from four points to six points is
pictured below.
\begin{equation}\label{randomtleqn}\CPic{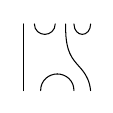} \in \TL(4,6)\end{equation}
When elements are represented by such pictures, the composition 
$$\TL(n,k) \otimes \TL(k,l) \to \TL(n,l) \conj{where} a\otimes b \mapsto ba$$
in the Temperley-Lieb category corresponds to vertical stacking.

\begin{definition}
  There are two operations relating different parts of $\TL$ that will be
  used repeatedly:
$$x \mapsto x \sqcup 1 \conj{ and } x \mapsto \bar{x}.$$
For each element $x \in \TL(n,k)$ there is an element
  $x\sqcup 1 \in\TL(n+1,k+1)$ obtained by placing a single vertical strand
  to the right of all of the diagrams appearing in the expression for $x$.  Given an
  element $x \in \TL(n,k)$ there is a corresponding element $\bar{x} \in
  \TL(k,n)$ obtained by flipping the diagrams representing $x$ upside
  down. Both of these operations are $q$-linear.
\end{definition}

\begin{definition}
The {\em Temperley-Lieb algebra} $\TL_n$ is given by the endomorphisms of the $n$th object in the Temperley-Lieb category:
$$\TL_n = \TL(n,n) = \Endo_{\quantsl}(V^{\otimes n}).$$ 
\end{definition}

\begin{definition}
The elements of $\TL_n$ are generated by {\em elementary diagrams} $e_i$ containing
$n-2$ vertical chords and two horizontal chords connecting the $i$th and the
$i+1$st positions. For instance,
$$e_1 =\CPic{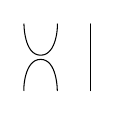} \in \TL_3.$$
\end{definition}

\begin{definition}\label{throughdegdef}
  Suppose that $a \in \TL(n,m)$ is a Temperley-Lieb diagram then there are
  many ways in which $a$ factors as a composition $a = cb$ where $b\otimes c
  \in \TL(n,l)\otimes \TL(l,m)$.  The \emph{through-degree} $\tau(a)$ of $a$
  is equal to the minimal $l$ achieved by such a factorization. If $a\in
  \TL(n,m)$ is a linear combination of Temperley-Lieb diagrams, $a = \sum_i
  f_i a_i$, then the through-degree of $a$ is defined by
$$\tau(a)=\max_i \{ \tau(a_i)\ \vert\ f_i \ne 0 \}.$$
\end{definition}

\begin{example}
The through-degrees $\tau$ of the two diagrams, \eqref{randomtleqn} and $e_1$, pictured above are two and one respectively.
\end{example}

\begin{remark}\label{REMREF}
  Through-degree can not increase when composing elements of $\TL$. In this
  manner, the category $\TL$ is filtered by through-degree. Let
  $\TL^k\subset \TL$ denote the subcategory consisting of morphisms that
  have a through-degree which is less than $k$:
$$\TL^k(n,m) = \{ f\in \TL(n,m) : \tau(f) < k \}.$$
Then there is a filtration
$$\cdots \subset \TL^{k-1} \subset \TL^{k} \subset \TL^{k+1} \subset \cdots\conj{ and } \TL = \cup_k \TL^k$$
which is respected by operations in the sense that:
$$\TL^k\ott \TL^k \xto{\circ} \TL^k \conj{ and } \TL^k\ott \TL^l \xto{\sqcup} \TL^{k+l}.$$
\end{remark}

\begin{remark}
  Instead of through-degree, one might choose instead to filter elements $a
  \in \TL_n$ by the number of turnbacks: 
$$\cap(a) = (n-\tau(a))/2.$$
This convention also appears in the literature. While Theorem
\ref{thmgenproj} and Definition \ref{ukgenprojdef} can be stated in terms of
turnbacks, it is awkward to use $\cap(a)$ for general elements of the
category $\TL$.
\end{remark}

\subsection{Idempotents from irreducibles}\label{idempotentsirred sec}
In this section we will explain the connection between $\quantsl$
representation theory and higher order Jones-Wenzl projectors. These
higher order Jones-Wenzl projectors will be explored in Section
\ref{genjwprojsec} and categorified in Section \ref{mainconstructionsec}.

The irreducible representations $V_k$ of $\quantsl$ are indexed by integers
$k\in\ZZ_{\geq 0}$. The trivial representation $V_0$ is 1-dimensional and
the fundamental representation $V_1 \cong V$ is 2-dimensional. In general,
we can use the {\em Clebsch-Gordan} rule:
\begin{equation}\label{freestrandtensor}
V_n \otimes V_1 \cong V_{n+1} \oplus V_{n-1}
\end{equation}
to decompose the tensor product $V^{\otimes n}$ into a direct sum of irreducible
representations:
$$V^{\otimes n} \cong V_n \oplus m_{n-2} V_{n-2} \oplus m_{n-4} V_{n-4} \oplus \cdots.$$
For each summand $m_{k} V_{k} \subset V^{\otimes n}$, there are
equivariant projection and inclusion maps,
$$V^{\otimes n} \xto{\pi_{n,k}} m_{k} V_{k} \xto{i_{n,k}} V^{\otimes n}$$
and Equation \eqref{sl2interp} implies the existence of a Temperley-Lieb element $p_{n,k} = i_{n,k}\circ \pi_{n,k}\in\TL_n$.

\begin{definition}\label{gen JW idempotents}
The $k$th \emph{higher order Jones-Wenzl projector} is the idempotent element $p_{n,k}\in\TL_n$ corresponding to the summand $m_k V_k \subset V^{\ott n}$.
\end{definition}

In the remainder of this section we will provide several descriptions of
these idempotents. A different discussion can be found in \cite{C, R2}.

\subsection{Jones-Wenzl projectors}\label{JW proj}
The Jones-Wenzl Projectors $p_n \in \TL_n$ are a special case of Definition
\ref{gen JW idempotents}.  The largest irreducible summand $V_n \subset
V^{\ott n}$ occurs with multiplicity one: $m_n = 1$. Since the Jones-Wenzl
projectors correspond to projection onto this summand,
$$V^{\otimes n} \to V_n \to V^{\otimes n},$$
they correspond to the higher order Jones-Wenzl projector of largest degree: $p_n = p_{n,n}$. In \cite{WN}, Wenzl introduced projectors using the recurrence relation, $p_1 = 1$ and
\begin{align}\label{jwrecureqn}
p_n &= p_{n-1}\sqcup 1 - \frac{[n-1]}{[n]} (p_{n-1}\sqcup 1) e_{n-1} (p_{n-1}\sqcup 1)
\end{align}
where the quantum integer $[n]$ is defined to be the Laurent polynomial
$$[n] = \frac{q^n - q^{-n}}{q-q^{-1}} = q^{-(n-1)} + q^{-(n-3)} + \cdots +  q^{n-3} + q^{n-1}.$$
If we depict $p_n$ graphically by a box with $n$ incoming and $n$ outgoing
chords
$$p_n = \CPic{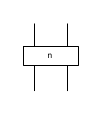}$$
then Formula \eqref{jwrecureqn} can be illustrated using diagrams:
$$\CPic{pn} = \CPic{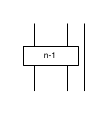} \,-\, \frac{[n-1]}{[n]}\! \CPic{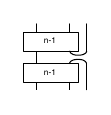}.$$

It can be shown that the Jones-Wenzl projectors are uniquely characterized by
the following properties:
\begin{enumerate}
\item $p_n \in \TL_n$. 
\item $p_n - 1$ belongs to the subalgebra generated by
  $\{e_1,e_2,\ldots,e_{n-1}\}$
\item $e_i p_n = p_n e_i = 0$ for all $i = 1, \ldots, n-1$.
\end{enumerate}
For more information see \cite{WN,KL,CK}.

\begin{remark}\label{remarkring}
Although the coefficient ring $\CC(q)$ is used throughout Section \ref{TL
  section}, we will consistantly interpret expressions like $[n]/[n+1]$ as a
power series in the ring $\ZZ[q^{-1}][[q]]$, see \cite{CK,CK2} for
discussion.
\end{remark}

\subsection{Higher order Jones-Wenzl projectors}\label{genjwprojsec}
Recall from Section \ref{idempotentsirred sec} that the $\quantsl$ representation $V^{\otimes n}$ decomposes as a sum of irreducible representations $V_{k}$,
$$V^{\otimes n} \cong \bigoplus_{k} m_{k} V_{k}.$$
Each summand $m_{k} V_{k} \subset V^{\otimes n}$ is a sum
of $m_{k}$ copies of the $k$th irreducible representation $V_{k}$. For each higher order projector
$p_{n,k}$ we would like to write an equation of the form
\begin{equation}\label{refinedidempotent}
p_{n,k} = \sum_{\e \in \aL_{n,k}} p_{\epsilon}
\end{equation}
where $\aL_{n,k}$ is a set indexing copies of $V_k$ in the summand $m_k
V_k\subset V^{\ott n}$ so that $|\aL_{n,k}| = m_{k}$. Each element $p_{\epsilon}\in \TL_n$
corresponds to projection onto a distinct copy of $V_{k}$ within $m_{k} V_{k}$. This
notation is introduced by the definition below.

\begin{definition}\label{domOrderstuff}
A \emph{sequence} $\epsilon$ is an $n$-tuple,
$$\epsilon = (i_1,i_2,\ldots,i_n) \quad\normaltext{ where }\quad i_k \in \{-1,1\} \normaltext{ for } 1\leq k \leq n.$$
The \emph{length} $l(\e)$ of a sequence $\epsilon = (i_1,i_2,\ldots,i_n)$ is given by $l(\e) = n$ and the \emph{size} $|\e|$ of $\e$ is defined to be the sum: $|\e|=i_1+\cdots+i_n$. 
If $\e = (i_1, i_2,\ldots,i_n)$ and $\delta = (j_1,j_2,\ldots,j_n)$ are two sequences then  $\e \dominates \delta$ when
$$i_1 + \cdots + i_k \geq j_1 + \cdots + j_k \normaltext{ for all } k = 1,\ldots, n.$$
For each sequence $\e$, if we denote by $0$ the $l(\e)$-tuple consisting
entirely of zeros then the sequence $\epsilon$ is \emph{admissible} if
$\epsilon \dominates 0$.

We denote by $\aL_n$ the {\em collection of all admissible sequences of length
$n$} and by $\aL_{n,k}\subset \aL_n$ the {\em collection of admissible sequences of
length $n$ and size $k$}.
$$\aL_n = \{ \e : l(\e) = n, \e\dominates 0 \}\conj{and} \aL_{n,k} =\{ \e \in\aL_n : |e| = k\}$$
The relation $\dominates$ when applied to these sets is called the
\emph{dominance order}.
\end{definition}

\begin{definition}
If $\e = (i_1,i_2, \ldots,i_n)$ is a sequence then we will use the notation
$\e \cdot (+1)$ and $\e \cdot (-1)$ to denote the sequence obtained from
$\e$ by appending $+1$ or $-1$ respectively.
$$\e\cdot (+1) = (i_1,i_2, \ldots,i_n,+1)\quad\normaltext{ and }\quad \e\cdot (-1) = (i_1,i_2, \ldots,i_n,-1)$$
\end{definition}

Associated to each $\e\in\aL_n$ is a special element $q_\e\in\TL_n$. Since
these special elements are vertically symmetric, it is easiest to define the
top half $t_\e$ first.

\begin{definition}\label{qedef}
If $\e \in \aL_n$ and $|\e| = k$ then there is an element  $t_\e \in \TL(k,n)$ defined inductively by $t_{(1)} = 1$, 
$$t_{\e\cdot (+1)} = \CPic{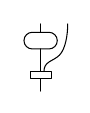} \quad\normaltext{ and }\quad t_{\e\cdot (-1)} = \CPic{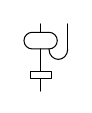}$$
where the box represents a Jones-Wenzl projector $p_k$ and the
marshmallow-shaped region represents the element $t_\e$. The special element $q_\e\in \TL_n$ is equal to the top $t_\e$ composed with its reverse:
$$q_\e = t_\e \bar{t}_\e.$$
The relation $p_kp_k = p_k$ allows us to eliminate one of the two central $p_k$ in the definition of $q_\e$.
\end{definition}

The elements $q_\e$ satisfy a recurrence relation.

\begin{lemma}
For each sequence $\e\in\aL_{n,k}$ there is a recurrence relation:
\begin{equation}\label{qerecureqn}
q_\e \sqcup 1 = q_{\e\cdot (+1)} + \frac{[k]}{[k+1]} q_{\e \cdot (-1)}
\end{equation}
\end{lemma}
\begin{proof}
  This follows from applying Equation \eqref{jwrecureqn} in Section \ref{JW
  proj} to the middle Jones-Wenzl projector $p_k$ of $q_\e$ in the definition above.
\end{proof}

We will use the special elements $q_\e \in \TL_n$ to construct idempotents
corresponding to the decomposition described in Section
\ref{idempotentsirred sec}. The following two propositions tell us that
there are scalars $f_\e \in \CC(q)$ such that the collection $p_\e
= f_\e q_\e$ satisfies
\begin{enumerate}
\item $p_\e p_\nu = \delta_{\e,\nu} p_\e$
\item $1_n = \sum_{\e\in\aL_n} p_\e$
\end{enumerate}
where $1_n\in\TL_n$ is the identity element. The first proposition below
tells us that composing $p_\e$ and $p_\nu$ when $\e \neq \nu$ yields
zero. Theorem \ref{decompositionOfIdentity} will address the second
equation and the first equation when $\e = \nu$.

\begin{proposition}\label{orthogonality}
The special elements $q_\e\in\TL_n$ defined above are mutually orthogonal. 
$$q_\e q_\nu = 0 \quad \normaltext{ for } \quad \e\neq \nu$$
\end{proposition}

\begin{proof}
Using the definition of $q_\e$ found above we can write $q_\e = a \bar{a}$
and $q_\nu = b \bar{b}$. If $\e\neq \nu$ then $a=a' p_k$ and $b = b' p_l$
where $k\ne l$. The product $q_\nu q_\e$ contains $p_k \bar{a}' b' p_l$ which is
equal to zero. By symmetry, $q_\e q_\nu$ also vanishes.
\end{proof}

In the next proposition, we show that, for each $\e\in \aL_n$, there are
constants $f_\e\in \CC(q)$ and idempotents $p_\e = f_\e q_\e$
which yield the decomposition of identity $1_n\in\TL_n$ mentioned above.

\begin{theorem}\label{decompositionOfIdentity}
For each $\e\in\aL_n$, there are idempotents $p_\e \in \TL_n$ which satisfy,
$$1_n = \sum_{\e\in\aL_n} p_\e.$$
Moreover, $p_\e = f_\e q_\e$ for some non-zero scalar $f_\e \in\CC(q)$.
\end{theorem}

\begin{proof}
The proof is by induction on the number of strands $n$.  When $n$ is $1$ the
only sequence is $\e = (1)$; we set $f_\e = 1$ so that $p_\e = q_\e = 1$.

Assume that there is a decomposition of $1_{n-1}$ and place a disjoint
strand next to everything. We have
$$1_{n-1} \sqcup 1 = 1_n = \sum_{\e\in \aL_{n-1}} f_\e\; q_\e \sqcup 1 = \sum_k \sum_{\e\in \aL_{n-1,k}} f_\e\;q_\e \sqcup 1,$$
in which the elements $p_\e = f_\e q_\e$ are idempotent. The
recurrence relation \eqref{qerecureqn} implies that
\begin{equation}\label{decompositionStrands2}
1_{n} = \sum_k \sum_{\e\in \aL_{n-1,k}} f_\e \left(q_{\e\cdot (+1)} + \frac{[k]}{[k+1]} q_{\e \cdot (-1)}\right).
\end{equation}
Setting $p_{\e\cdot (+1)} = f_\e q_{\e\cdot (+1)}$ and $p_{\e\cdot (-1)}=f_\e
\frac{[k]}{[k+1]} q_{\e\cdot (-1)}$ yields the equation in the statement of
this proposition.

To show that the $p_\nu$ are idempotent, multiply $1_n = \sum_{\e} p_\e$ on
the left with $p_\nu$. The previous proposition implies that
$$p_\nu=\sum_{\e}p_\nu p_\e = p_\nu p_\nu.$$
\end{proof}

\begin{remark}
  By Proposition \ref{orthogonality}, $p_\e p_\nu = 0$ when $\e \ne \nu$
  because the projectors $p_\e$ differ from the elements $q_\e$ by scalars.
  The construction in the proof of Theorem \ref{decompositionOfIdentity}
  above implies the equation below.
\begin{equation}\label{peeqn}
p_\e \sqcup 1 = p_{\e\cdot (+1)} + p_{\e\cdot (-1)}
\end{equation}
By convention $p_\e = 0$ when $\e \not\in\aL_n$. This equation corresponds to
the Clebsch-Gordan rule \eqref{freestrandtensor} in Section
\ref{idempotentsirred sec}. 
\end{remark}

\begin{remark}Each idempotent $p_\e$ corresponds to projection onto one term $V_k \subset
m_k V_k\subset V^{\otimes n}$. The higher order projectors $p_{n,k}$
correspond to the entire subspace $m_k V_k \subset V^{\otimes n}$.
\end{remark}

\begin{definition}\label{generalizedprojdef}
The $k$th \emph{higher order Jones-Wenzl projector} $p_{n,k} \in \TL_n$ is given by the sum,
$$p_{n,k} = \sum_{\e \in \aL_{n,k}} p_\e.$$
\end{definition}

\begin{remark}
From Proposition \ref{orthogonality} and Theorem \ref{decompositionOfIdentity} it
follows that the elements $p_{n,k}\in\TL_n$ form a system of
mutually orthogonal idempotents. This means that 
$$p_{n,k} p_{n,l} = \delta_{k,l} p_{n,k} \conj{ and } 1_n = \sum_{k} p_{n,k}.$$
\end{remark}

\begin{remark}\label{factorobs}
Since $p_{n,k}$ is a sum of elements $p_\e$ with $|\e| = k$ and each $p_\e$
necessarily factors through a Jones-Wenzl projector $p_k$, the projector
$p_{n,k}\in\TL_n$ is a linear combination of diagrams which
factor as $b p_{k} a$ where $a\otimes b\in \TL(n,k)\otimes\TL(k,n)$.
\end{remark}

Although a definition of the higher order projectors $p_{n,k}$ was given in
Section \ref{idempotentsirred sec}, it is often useful to characterize
elements {\em intrinsically} in terms of their interaction with other
elements and gluing operations. This is the definition given below and the
one which will lift to the categorical setting in Section \ref{charprojsec}.

\begin{theorem}\label{thmgenproj}
The higher order Jones-Wenzl projectors
$$p_{n,k}\in\TL_n$$
of Definition \ref{generalizedprojdef} are characterized uniquely by the following properties
\begin{enumerate}
\item The through-degree $\tau(p_{n,k})$ of $p_{n,k}$ is equal to $k$.

\item The projector $p_{n,k}$ vanishes when the number of turnbacks is sufficiently high. For each $l\in\ZZ_+$ and $a \in \TL(n,l)$ if $\tau(a) < k$ then
$$a p_{n,k} = 0 \quad\normaltext{ and }\quad p_{n,k} \bar{a} = 0.$$ 

\item The projector $p_{n,k}$ fixes elements of through-degree $k$ up to lower through-degree terms. For each $l\in\ZZ_+$ and $a\in \TL(n,l)$ if $\tau(a) = k$ then
$$a p_{n,k} = a + b$$
where $\tau(b)<k$.
\end{enumerate}
\end{theorem}

In essence, these properties state that the projectors $p_{n,k}$ control and
respect the filtration of $\TL$ by through-degree $\tau$, see also the
discussion following Definition \ref{throughdegdef}.

\begin{proof}
We begin by proving that the elements $p_{n,k}$ defined above satisfy
properties (1)--(3). Using Remark \ref{factorobs} above, we can write
$p_{n,k}$ as a sum of the form,
$$p_{n,k} = \sum_{x,y} x p_k \bar{y}.$$
The first property follows from $\tau(p_k) = k$. Now pick some $l \in\ZZ_+$ and $a\in \TL(n,l)$. 

For the second property, if we assume that $\tau(a)<k$ then
$$a p_{n,k} = \sum_{x,y} a x p_k \bar{y} = 0$$
since $\tau(a d) \leq \tau(a) < k$ and $p_k$ kills diagrams of through-degree less than $k$.  For the same reason $p_{n,k}\bar{a} =0$. 

For the third property, if we assume that $\tau(a) = k$ then
\[
a = a 1_n = \sum_{l} a p_{n,l} = \sum_{l\leq k} a p_{n,l},
\]
so that rearranging terms gives $a p_{n,k}  = a - \sum_{l<k} a p_{n,l}$.

Suppose that $e\in\TL_n$ satisfies Properties (1)--(3) above, we will show
that $e = p_{n,k}$.
If $l<k$ then Property (2) for $e$ implies that $e p_{n,l} =0$.  If $l>k$
then Property (2) for $p_{n,l}$ implies that $e p_{n,l} =0$. Therefore,
$$ e = e 1_n = \sum_{l} e p_{n,l} = e p_{n,k}.$$
Property (3) implies that $e p_{n,k} = p_{n,k} + b$ where $\tau(b) < k$.  By Theorem \ref{decompositionOfIdentity},
$$e = e p_{n,k} = e p_{n,k}^2 = (p_{n,k}+b)p_{n,k} = p_{n,k}^2 = p_{n,k}.$$
\end{proof}

We continue our discussion of the higher order projectors with a series of
observations.

\newcommand{\ptwT}[0]{\frac{1}{2} k^2}
\newcommand{\ptwQ}[0]{k}
\newcommand{\ptwS}[0]{}

\newcommand{\ptwTsq}[0]{k^2}
\newcommand{\ptwQsq}[0]{2k}
\newcommand{\ptwSsq}[0]{}

\newcommand{\ntwT}[0]{-\frac{1}{2} k^2}
\newcommand{\ntwQ}[0]{-k}
\newcommand{\ntwS}[0]{}

\newcommand{\ntwTsq}[0]{-k^2}
\newcommand{\ntwQsq}[0]{-2k}
\newcommand{\ntwSsq}[0]{}
\newcommand{\pSHIFT}[0]{t^{\ptwT} q^{\ptwQ}}
\newcommand{\nSHIFT}[0]{t^{\ntwT} q^{\ntwQ}}

\newcommand{\DpSHIFT}[0]{(-1)^k q^{\frac{1}{2}k(k+1)}}
\newcommand{\DnSHIFT}[0]{(-1)^k q^{-\frac{1}{2}k(k+1)}}

\begin{proposition}\label{slidethroughobs}
For each element $a\in \TL(n,m)$, the equation $a p_{n,k} = p_{m,k} a$ holds.
$$\CPPic{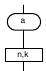} = \hspace{.5in}\CPPic{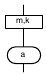}$$
\end{proposition}
\begin{proof}
Recall that $1_m = \sum_l p_{m,l}$ and $1_n = \sum_l p_{n,l}$.  Simplifying
the resulting expressions for $1_m a p_{n,k}$ and $p_{m,k} a 1_n$ gives $a
p_{n,k} = p_{m,k} a p_{n,k} = p_{m,k} a$.
\end{proof}

\begin{corollary}
The higher order projectors $p_{n,k}$ are contained in the center, $p_{n,k}\in \eZ(\TL_n)$, of the Temperley-Lieb algebra.
\end{corollary}

\begin{corollary}
If $D\in\TL_n$ and the through-degree $\tau(D) = l$ so that
$D = ba$ where $a\otimes b \in \TL(n,l)\otimes\TL(l,n)$ then
$$p_{n,k} D = p_{n,k} b a = b p_{l,k} a.$$
\end{corollary}

\begin{remark}
This means we can slide a $p_{n,k}$ past some turnbacks onto a fewer number of
strands as long as we change it to a $p_{l,k}$. In pictures,
$$\CPic{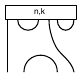} = \;\;\;\; \CPic{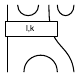}.$$
\end{remark}

We conclude this section with another definition of $p_{n,k}$. This
definition has the value of expressing $p_{n,k}$ in terms of simpler
projectors.

\begin{proposition}\label{indpkobs}
The higher Jones-Wenzl projector $p_{n,k}$ satisfies the recurrence relation:
$$\CPic{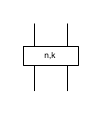} = \;\;\;\; \sum_{j=1}^{n-1} \frac{[j]}{[j+1]} \CPic{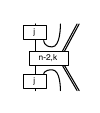}$$
\end{proposition}
\begin{proof}
Equation \eqref{jwrecureqn} implies that
$$p_n = 1_n - \sum_{j=1}^{n-1} \frac{[j]}{[j+1]} (p_j\sqcup 1_{n-j}) e_j (p_j\sqcup 1_{n-j}).$$
Applying $p_{n,k}$ to both sides of this equation gives:
$$p_{n,k} = \sum_{j=1}^{n-1} \frac{[j]}{[j+1]} p_{n,k} (p_j\sqcup 1_{n-j}) e_j (p_j\sqcup 1_{n-j}),$$
which becomes the desired equation after applying Proposition \ref{slidethroughobs}.
\end{proof}

Iterating this formula expresses $p_{n,k}$ purely in terms of Jones-Wenzl projectors $p_l$.

\section{Categorification of the Temperley-Lieb category} \label{categorified TL section}
In this section we recall Dror Bar-Natan's graphical formulation \cite{DBN}
of the Khovanov categorification \cite{Kh00, KH}. We follow the same
conventions as \cite{CK}.

There is a pre-additive category $\PCob(n)$ whose objects are isotopy
classes of formally $q$-graded Temperley-Lieb diagrams with $2n$ boundary
points. The morphisms are given by $\mathbb{Z}[\alpha]$-linear combinations
of isotopy classes of orientable cobordisms, decorated with dots, and bounded
in $D^2\times [0,1]$ between two disks containing such diagrams. The
\emph{degree} of a cobordism $C : q^i A \to q^j B$ is given by
$$\deg(C) = \deg_{\chi}(C) + \deg_q(C)$$
where the topological degree $\deg_{\chi}(C) = \chi(C) - n$ is given by the
Euler characteristic of $C$ and the $q$-degree $\deg_q(C) = j - i$ is given
by the relative difference in $q$-gradings.  The maps $C$ used throughout
the paper will satisfy $\deg(C) = 0$. The formal $q$-grading will be chosen
to cancel the topological grading. 

When working with chain complexes, every object will also contain a
homological grading and every map will have an associated homological
degree. This homological degree, or $t$-degree, is not part of the definition
$\deg(C)$. We may refer to degree as internal degree in order to
differentiate between degree and homological degree.

We impose the relations below to obtain a new category $\Cob(n)$ as a
quotient of the category $\Mat(\PCob(n))$ formed by allowing direct sums of
objects and maps between them.
$$\CPic{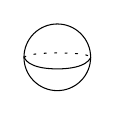} = 0 \hspace{.75in} \CPic{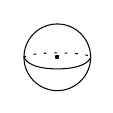} = 1 \hspace{.75in} \CPic{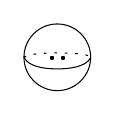} = 0 \hspace{.75in} \CPic{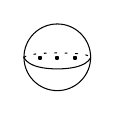} = \alpha$$
$$\CPic{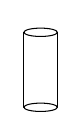} = \CPic{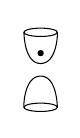} + \CPic{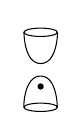}$$
The dot is determined by the relation that two times a dot is equal to a
handle. When $\alpha = 0$ the cylinder or neck cutting relation implies that closed surfaces
$\Sigma_g$ of genus $g > 3$ evaluate to $0$. In what follows we will let
$\alpha$ be a free variable and absorb it into our base ring ($\Sigma_3=
8\alpha$). One can think of $\alpha$ as a deformation parameter, see
\cite{DBN}.

The categories $\Cob(n)$ fit together in much the same way as the
Temperley-Lieb algebras $\TL_n$. There is an inclusion 
$$ - \sqcup 1_{m-n} : \Cob(n) \to \Cob(m) \conj{ when } n < m$$ 
which is obtained by placing $m-n$ disjoint vertical
line segments to the right to each object.  On morphisms $f\sqcup 1_{m-n}$ is defined to be the union of $f$ and $m-n$ copies of an identity cobordism. If $m=n$ then the empty set is used.

There is a category $\Cob(m,n)$ with objects corresponding to diagrams in 
$\TL(m,n)$, so that $\Cob(n) = \Cob(n,n)$. There is a composition
$$\otimes : \Cob(n,k) \times \Cob(k,l) \to \Cob(n,l)\conj{where} A\times B \mapsto B\otimes A$$
obtained by gluing all diagrams and morphisms along the $k$ boundary points
and $k$ boundary intervals respectively. Pictorially,
\begin{equation}\label{otimeseqn}C\otimes D = \CPic{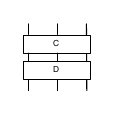}.\end{equation}
This composition makes the collection of categories $\Cob(n,k)$ into a
2-category $\Cob$. The relationship between $\Cob$ and the Temperley-Lieb
category $\TL$ can be described using the Grothendieck group functor $\K_0$.

\begin{theorem}
The 2-category $\Cob$ categorifies the Temperley-Lieb category $\TL$. There are isomorphisms,
$$\TL(n,k) \cong \K_0(\Cob(n,k)) \otimes_{\ZZ[q,q^{-1}]} \CC(q) $$
which commute with composition.
\end{theorem}

These isomorphisms commute with the compositions \eqref{otimeseqn}. For more detail see \cite{DBN,KH, CK2}.

\begin{definition}
The category of {\em chain complexes of cobordisms} will be denoted by $\Kom(n,m)$.
$$\Kom(n,m) = \Kom(\Cob(n,m))$$
\end{definition}

Unless otherwise stated chain complexes are bounded from below in
homological degree. All chain complexes produced in what follows will have
differentials with components having internal degree zero. Restricting to
the subcategory of chain complexes with degree zero differentials yields a
well-behaved Grothendieck group $\K_0(\Kom(n,k))$, see \cite{CK}.

The category of chain complexes can be enriched to form a differential
graded category.

\begin{definition}\label{komstardef}
There is a differential graded category, $\Kom^*(n,m)$, which has the same
objects as the category $\Kom(n,m)$ but with morphisms that are given by allowing maps of all
homological degrees.
\end{definition}

If $f \in \Morph^m(A, B)$ is a map of homological degree $m$ then $d(f) =
[d, f] \in \Morph^{m+1}(A, B)$. Each collection $\Hom^*(A, B)$, of morphisms
from $A$ to $B$ in $\Kom^*(n,m)$, is a chain complex and the differential
$d$ is a derivation with respect to composition of maps.

\subsection{Grading shifts}
In this section we remind the reader how degree shifts are denoted. Each
chain complex can be shifted in $q$-degree or $t$-degree.

If $A$ is a chain complex then $t A$ will denote the chain complex shifted
in homological degree by $1$,
$$(t A)_i = A_{i-1} \quad\normaltext{ and }\quad d_{tA} = -d_A$$
We will use $q A$ to denote the chain complex satisfying $\deg_q(qB) =
\deg_q(B) + 1$ where $B\in\PCob(n)$ corresponds to a summand of $A$, see
Section \ref{categorified TL section} for a discussion of $q$-degree.

If $C\in\Kom(n,m)$ is a chain complex and $f(q)\in\ZZ[q^{-1}][[q]]$ is a
power series then we will write $\left[ f(q) C\right]$ for an iterated
cone of chain complexes $A_0, A_1,\ldots$ in which $A_i = C$ for all $i$.
The relation
$$[ f(q) C ] = f(q)\cdot [C].$$
holds in the Grothendieck group $\K_0(\Kom(n,m))$.

\begin{definition}\label{boxnotation}
If $C$ is a chain complex of the form $Q_\e$ where $\K_0(Q_\e) = q_\e$ then
we will consistently omit a product of terms of the form $[k]/[k+1]$ from
the bracket notation. Usually,
$$[Q_\e] = [f_\e Q_\e]$$
where $f_\e$ is defined in the proof of Theorem \ref{decompositionOfIdentity}.
\end{definition}

\subsection{Universal projectors}
The most important object in the categories defined above is an idempotent
chain complex $P_n\in\Kom(n)$ which categorifies the Jones-Wenzl projector
$p_n\in \TL_n$.  The chain complexes $P_n$ will be used repeatedly in later
sections in order to construct the projectors $P_\e$ and $P_{n,k}$ that
correspond to the elements $p_\e$ and $p_{n,k}$ introduced in Section
\ref{TL section}.

\begin{theorem}{(\cite{CK})}\label{uprojectorthm}
There exists a chain complex $P_n \in \Kom(n)$ called the \emph{universal
  projector} which satisfies
\begin{enumerate}
\item $P_n$ is positively graded with differential having internal degree zero.
\item The identity diagram appears only in homological degree zero and only
  once.
\item For each diagram $D$ which is not identity, the chain complex $P_n \ott D$ is contractible.
\end{enumerate}
These three properties characterize $P_n$ uniquely up to homotopy.
\end{theorem}
See also \cite{FSS, R1}. See \cite{Cautis} and \cite{BenW2} for related
ideas. 

We conclude this section with a lemma. This lemma plays an important role in the proof of Theorem \ref{conethingthm}.
\begin{lemma} \label{standardformlemma}
If $P_n \in \Kom(n)$ is a projector then there is a twisted complex
$$\CPic{pn} = \Cone\left(\hspace{-.02in}\CPic{pnm1parjw}\hspace{-.05in} \to t \Bigg[\frac{[n-1]}{[n]}\! \CPic{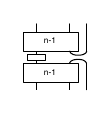} \Bigg]\right)$$
which is also a projector.
\end{lemma}
\begin{proof}
The proof follows from tensoring the Frenkel-Khovanov complex of \cite{CK} for $P_n$ with $(P_{n-2} \sqcup 1_2) \otimes e_{n-1} \otimes (P_{n-1} \sqcup 1_1)$ then contracting portions of the subcomplex consisting of projectors containing turnbacks.
\end{proof}

\section{Twisted complexes and operations on twisted complexes}

\subsection{Twisted complexes}\label{twistedcomsec}
In this section we recall the definition of the category $\Tw \aA$ of
twisted complexes over a differential graded category $\aA$, see
\cite{BK,OTT}. The reader may assume that $\aA = \Kom^*(n,m)$, see
Definition \ref{komstardef}.

Our main construction in Section \ref{mainconstructionsec} will occur in the
category of twisted complexes. Informally, the definitions in this section
codify situations in which the objects of study are chain complexes $M$ with
a decreasing filtration
$$M = F^0 M\supseteq F^1 M \supseteq F^2 M \supseteq \dots $$ 
and a splitting $F^{i+1}(M) = G_{i+1} \oplus F^i(M)$ as graded objects. Maps
are required to respect this filtration.

The definitions presented here are variations on standard ones which allow
one to work with categories of twisted complexes that are unbounded and
indexed by countable sets (such as $\ZZ_+$). This is accomplished
by requiring that maps are lower triangular, see Definition \ref{totdef}
below.

\begin{definition}
A \emph{twisted complex} over $\aA$ is a collection
$$\{ (E_i), q_{ij} : E_i \to E_j \} \conj{where} i\in \ZZ_{+}$$
consisting of objects $E_i \in \aA$ and maps $q_{ij}$ of degree $1$ which satisfy $q_{ij} = 0$ for $i \geq j$ and the equation
\begin{equation}\label{MCeqn}
(-1)^j d_{\aA} (q_{ij}) + \sum_{k} q_{kj} \circ q_{ik} = 0.
\end{equation}
\end{definition}

\begin{definition}
  Twisted complexes which satisfy the condition that $q_{ij} = 0$ when $i
  \geq j$ are called \emph{one-sided}. 
\end{definition}

\begin{remark}
  All infinite twisted complexes will be one-sided. The resolution of
  identity $R_n$ in Theorem \ref{pespexistencethm} is a one-sided twisted
  complex.
\end{remark}

It will be convenient later to use ordered sets besides $\ZZ_+$ to index
components of twisted complexes. In particular, the set of sequences $\aL_n$
together with the dominance order described in Definition
\ref{domOrderstuff} will be used throughout
Section \ref{mainconstructionsec}. In general, it will be clear from context
when this is done.

\begin{definition}{($\Tw\aA$)}
The one-sided twisted complexes form a differential graded category.  If $A
= \{(A_i), a_{ij}\}$ and $B = \{(B_i), b_{ij}\}$ then degree $k$ maps are
those that intertwine the diagrams formed by $A$ and $B$,
$$\Morph^k_{\Tw\aA}(A,B) = \prod_{i\leq j} \Morph^{k+i-j}_{\Kom^*(n,m)}(A_i, B_j).$$
In other words, morphisms $f : A \to B$ are collections $\{f_{ij}\}$ of maps
having the appropriate degree which satisfy $f_{ij} = 0$ unless $i\leq
j$. The composition of morphisms is defined in terms of components by the
equation,
$$(f\circ g)_{ij} = \sum_{i\leq k \leq j} f_{kj} \circ g_{ik}.$$
If $f\in\Hom^*(A,B)$ is given by $\{f_{ij}\}$ then the equation
$$(d f)_{ij} = (-1)^j d_{\aA} (f_{ij}) + \sum_k b_{kj} \circ f_{ik} - (-1)^{\vert f\vert} f_{kj}\circ q_{ik}$$
determines a differential which makes $\Tw\aA$ into a differential graded category.
\end{definition}

The categories $\Tw\Kom^*(n,m)$ are examples of pre-triangulated categories.
Pre-triangulated categories can be seen as an alternative to triangulated
categories because every such category $\aA$ yields a triangulated category
$H^0(\aA)$, see \cite{BK}.

If $\{E_i\}\subset\Kom(\aC)$ is a collection of non-negatively graded chain
complexes then as graded objects, $\prod_{i\geq 0} t^i E_i \cong
\bigoplus_{i\geq 0} t^i E_i$ since the direct product is finite in each
degree. This allows us to flatten each twisted complex $A = \{(A_i),
a_{ij}\}$ to a chain complex $\Tot(A)$ by summing together the individual
components $A_i$ of $A$.

\renewcommand{\matrix}[1]{\begin{bmatrix}#1\end{bmatrix}}
\begin{definition}\label{totdef}
If $\aA = \Kom(\aC)$ is the category of non-negatively graded chain complexes over an additive category $\aC$  then there is a dg functor 
$$\Tot:\Tw\aA \rightarrow \aA$$
from twisted complexes to complexes defined on objects $\{(E_i), q_{ij}\}\in\Tw\aA$ by
\[
\Tot(\{(E_i), q_{ij}\}) = \left\{\bigoplus_{i\geq 0}t^i E_i\ ,\  d\right\} \conj{where} d = \left(\begin{array}{cccc} d_{E_0} & & &  \\ q_{01} & -d_{E_1} & & \\ q_{02} & q_{12} & d_{E_2}  & \\ \vdots & \vdots & \vdots  &\ddots \end{array} \right) 
\]
and on morphisms $f = \{f_{ij}\}$ by
\[
\Tot(f) = \left(\begin{array}{cccc} f_{00} & & & \\ f_{01} & f_{11} & & \\ f_{02} & f_{12} & f_{22}  & \\ \vdots & \vdots & \vdots  &\ddots \end{array} \right).
\]

\end{definition}

\begin{remark}
The condition that $d_{\Tot(A)}^2 = 0$ is implied by Equation \eqref{MCeqn} above.  
\end{remark}

\begin{remark}
The functor $\Tot$ defined above preserves homotopy equivalences. In
particular, for all $X,Y\in\Tw\aA$,  $X\simeq Y \Rightarrow \Tot(X)\simeq \Tot(Y)$.
\end{remark}

\begin{remark}\label{totremark}
If $\aC$ is an additive category which contains countable direct products
and $\aA$ is the differential graded category of possibly unbounded chain
complexes over $\aC$ then we can define a dg functor 
$$\Tot^\Pi : \Tw\aA \rightarrow \aA$$
using the formulas above with $\oplus$ replaced by $\Pi$.
\end{remark}

\begin{definition}{(Convolution)}\label{convolutiondef}
If a chain complex $A$ is the total complex of some twisted complex $\{(E_i),
q_{ij}\}$ then we say $A$ is a \emph{convolution} of $\{ (E_i), q_{ij}\}$.
\end{definition}

\begin{example}\label{twistedcomplexexample}
In Lemma \ref{combinTheHairs} the twisted complex $T$ pictured below is considered.
\begin{center}
\begin{tikzpicture}[scale=10, node distance=2.5cm]
\node (A1) {$A$};
\node (B1) [right=1cm of A1] {$t B$};
\node (C1) [right=1cm of B1] {$t^2 C$};
\node (D1) [right=1cm of C1] {$t^3 D$};
\draw[->] (A1) to node {$\alpha$} (B1);
\draw[->] (B1) to node {$\beta$} (C1);
\draw[->] (C1) to node {$\gamma$} (D1);
\draw[->, bend left=50] (A1) to node  {$\rho$} (C1);
\draw[->, bend right=50] (A1) to node [swap] {$\eta$} (D1);
\draw[->, bend left=50] (B1) to node {$\sigma$} (D1);
\end{tikzpicture}
\end{center}
Each object $A$, $B$, $C$ and $D$ is a chain complex. The convolution $\Tot(T)$ is the chain complex 
$A \oplus t B \oplus t^2 C \oplus t^3 D$
with differential
$$d_{\Tot(T)} = \left(\begin{array}{cccc} d_A & & & \\ \alpha & -d_B & & \\ \rho & \beta & d_C & \\ \eta & \sigma & \gamma & -d_D \end{array}\right).$$
This twisted complex is one-sided with respect to the order of the letters
appearing in the alphabet.
\end{example}

The notion of hull defined below formalizes the idea of the subcategory of
all chain complexes built out of iterated extensions of elements of some
fixed set of chain complexes. 

\begin{definition}\label{hulldef}
If $\aE = \{A_1,\ldots, A_r\} \subset \Kom(\aC)$ is a collection of chain
complexes then the \emph{hull} $\inp{\aE}\subset \Kom(\aC)$ is the smallest
strictly full additive subcategory containing each $A_i$ and closed under
convolution.

In particular, if $\{(E_i), q_{ij}\}\in\Tw\Kom(\aC)$ satisfies $E_i \in \inp{\aE}$ for all $i\in\ZZ_+$ then $\Tot(\{(E_i), q_{ij}\}) \in \inp{\aE}$.
\end{definition}


\begin{definition}\label{conedef}
Suppose that $A = \{(A_i), a_{ij}\}$ and $B = \{ (B_i), b_{ij} \}$ are twisted
complexes and $f = \{ f_{ij} \} : A \to B$ is a degree zero cycle then the
\emph{cone} of $f$ is the twisted complex given by
$$\Cone(f) = \left\{ (A_i\oplus B_{i-1})\ ,\ \left(\begin{array}{cc} a_{ij} & 0 \\ f_{i,j-1} & - b_{i-1,j-1} \end{array}\right)\right\}.$$
\end{definition}

\begin{remark}
The condition that $\Cone(f)$ is a twisted complex is equivalent to the requirement that $f$ is a degree zero cycle in the definition above.  
\end{remark}

\begin{remark}
If $A,B$ are twisted complexes over a category of chain complexes and $f :
A\rightarrow B$ a degree zero cycle then 
$$\Tot(\Cone(f)) = \Cone(\Tot(f)).$$
\end{remark}

\begin{definition}\label{truncdef}
Suppose that $Y=\{(Y_i), y_{ij}\}$ is a twisted complex indexed by $\ZZ_+$. For each $r,s\in\ZZ_+$ with $r\leq s$ the $[r,s]-$\emph{truncation} of $Y$ is given by
$$Y_{[r,s]} = \{(T_i), t_{ij}\}$$ 
where $T_i = Y_i$, $t_{ij} = y_{ij}$ when $i,j \in [a,b]$ and  $T_i = 0$, $t_{ij} = 0$ when $i,j \not\in [r,s]$.
\end{definition}

The lemma below says that a twisted complex is determined by its
truncations and each truncation is an iterated mapping cone. For simplicity
of notation, we restrict to $\ZZ_+$-indexed twisted complexes over the
categories $\Kom^*(n)$.

\begin{lemma}\label{truncObs}
If $\{(E_i), q_{ij}\}$ is a twisted complex over $\Kom^*(n)$ then for each integer $s\geq 0$,
\[
\Tot(Y_{[0,s]}) = \Cone(\Tot(Y_{[0,s-1]})\buildrel \d\over\rightarrow t^{s-1} E_s),
\]
where $\d = \left(\begin{array}{cccc} q_{0,s} & q_{1,s} & \dots & q_{s-1,s} \end{array}\right)$
is a chain map of degree zero. Conversely, if we have chain complexes $C_s$
and maps $\d_s:C_s\rightarrow t^s E_{s+1}$ such that $C_{s+1} = \Cone(\d_s)$
then there is a unique twisted complex $Y$ such that $C_s =
\Tot(Y_{[0,s]})$.
\end{lemma}

Recall that an object $E$ in a dg category $\aA$ is \emph{contractible} when
the map $\Id_{E}$ is a boundary in the mapping space $\Endo_{\aA}(E)$. The next lemma is a useful tool
for showing that certain filtered chain complexes are contractible.

\begin{lemma}\label{contractLemma}
  If $\{E_i\}\subset \aA$ is a collection of contractible objects then each
  twisted complex $\{(E_i),q_{ij}\}$ is contractible.
\end{lemma}

\begin{remark}
One subtlety to keep in mind is that the corresponding result for chain
complexes only holds in situations where convolution $\Tot$ is defined,
e.g. over a category of non-negatively graded chain complexes \emph{or} a
category of, possibly unbounded, chain complexes over an additive category
containing countable direct products.
\end{remark}

The following theorem says that the dg subcategory of $\aA$ determined by
the hull of $\aE$ is controlled by the dg algebra of $\Hom$-spaces between
objects in $\aE$.

\begin{theorem}{(\cite{BK})}\label{modulethmdef}
If $\aE$ is a collection of objects in a pre-triangulated category $\aA$ then the category of differential graded modules over the algebra
$$E = \bigoplus_{i,j} \Hom^*(E_i,E_j)$$
is equivalent to the category of $\inp{\aE}$.
\end{theorem}




\subsection{Operations on twisted complexes}\label{ellemmasec}
In this section we introduce a number of lemmas and notations which will be used repeatedly
throughout Sections \ref{explicitsec} and \ref{mainconstructionsec}.

We will use the following proposition in Theorem \ref{conethingthm}
to construct the chain complexes $P_\e$.

\begin{proposition}{(Obstruction Theory)}\label{obstprop}
For each pair of chain maps $\alpha : A \to B$ and $\beta : B \to C$ there is
a chain map $\gamma : \Cone(\alpha) \to tC$ of the form $\gamma =
\left(\begin{array}{cc} -h & \beta\end{array}\right)$ where
  $\beta\circ\alpha = d_C \circ h$if and only if $\beta \circ \alpha \simeq
  0$.
Moreover, if $\Hom^*(A,C)\simeq 0$ then $\beta\circ\alpha \simeq 0$ and the map $\gamma$ is unique up to homotopy.
\end{proposition}
\begin{proof}
The associated homotopy category $\Ho(\Kom)$ is triangulated. There is an exact triangle,
$A \to B \to \Cone(\alpha)$. Applying the functor $\Morph(-,C)$ yields a long exact sequence,
$$\cdots \to \Morph^i(\Cone(\alpha),C) \to \Morph^i(B,C) \to \Morph^i(A,C) \to \cdots. $$
If $\beta\circ\alpha \simeq 0$ then $\alpha^*(\beta) = 0$ and exactness
implies the existence of $\gamma$. One can check that $\gamma$ is given by
the map above between chain complexes. Uniqueness of $\gamma$ is implied by
exactness on the other side.
\end{proof}

\begin{remark}
The uniqueness of the lifts $\gamma$ in the proposition above will guarantee
that there is exactly one choice at each stage in the construction of the
projectors $P_\e$, see Theorem \ref{conethingthm} and Corollary
\ref{pespexistenceprop}.
\end{remark}

It will be useful to add a contractible chain complex to a chain complex
using the $\Cone$ construction.

\begin{lemma}{(Substitution)}\label{addcontractiblelemma}
Let $A,B,C,D$ be chain complexes and $f:B\rightarrow D$ a chain map. Then we have:
\begin{center}
\begin{tikzpicture}[scale=10, node distance=2.5cm]
\node (A1) {$A$};
\node (B1) [right=1cm of A1] {$B$};
\node (C1) [right=1cm of B1] {$C$};
\draw[->] (A1) to node {$\alpha$} (B1);
\draw[->] (B1) to node {$\beta$} (C1);
\draw[->, bend left=50] (A1) to node {$\sigma$} (C1);
\node (P1) [right=.25cm of C1] {$\simeq$};
\node (A2) [right=.25cm of P1] {$A$};
\node (B2) [right=1cm of A2] {$\Cone(f)$};
\node (C2) [right=1cm of B2] {$C$};
\node (D2) [below=1cm of B2] {$D$};
\draw[->, bend left=50] (A2) to node {$\sigma$} (C2);
\draw[->] (A2) to node {$\gamma$} (B2);
\draw[->] (B2) to node {$\delta$} (C2);
\draw[->, bend right=25] (A2) to node[swap] {$f\circ \alpha$} (D2);
\draw[right hook->] (D2) to node  {$\zeta$} (B2);
\end{tikzpicture}
\end{center}
where $\gamma = \left(\begin{array}{c} \alpha \\ 0 \end{array}\right)$, 
$\delta = \left(\begin{array}{cc} \beta & 0 \end{array}\right)$ and $\zeta = \left(\begin{array}{c} 0 \\ \Id \end{array}\right)$.
\end{lemma}

The following corollary is an inductive version of the previous lemma.

\begin{corollary}{(Inductive Substitution)}\label{substitution}
Suppose that $A = \{(A_i), q_{ij}\}$ is a one-sided twisted complex and there are maps $f_i:A_i\rightarrow C_i$. Then $A \simeq \{ (C_i \hookrightarrow  B_i), p_{ij} \}$ where $B_i = \Cone(f_i)$.
\end{corollary}

Note that we may have $C_i = 0$ for some $i$ in the above corollary.

The next lemma allows us to remove arrows between objects in a twisted
complex when the $\Hom$-spaces between these objects is contractible (as in
Theorem \ref{homComputation}).

\begin{lemma}{(Combing)}\label{combinTheHairs}
If $\beta \in \Morph^*(B,C)$ is a boundary then 
\begin{center}
\begin{tikzpicture}[scale=10, node distance=2.5cm]
\node (A1) {$A$};
\node (B1) [right=1cm of A1] {$B$};
\node (C1) [right=1cm of B1] {$C$};
\node (D1) [right=1cm of C1] {$D$};
\draw[->] (A1) to node {$\alpha$} (B1);
\draw[->] (B1) to node {$\beta$} (C1);
\draw[->] (C1) to node {$\gamma$} (D1);
\draw[->, bend left=50] (A1) to node  {$\rho$} (C1);
\draw[->, bend right=50] (A1) to node [swap] {$\eta$} (D1);
\draw[->, bend left=50] (B1) to node {$\sigma$} (D1);
\node (P1) [right=.25cm of D1] {$\cong$};
\node (A2) [right=.25cm of P1] {$A$};
\node (B2) [right=1cm of A2] {$B$};
\node (C2) [right=1cm of B2] {$C$};
\node (D2) [right=1cm of C2] {$D$};
\draw[->] (A2) to node {$\alpha$} (B2);
\draw[->] (C2) to node {$\gamma$} (D2);
\draw[->, bend left=50] (A2) to node  {$\rho + h\alpha$} (C2);
\draw[->, bend right=50] (A2) to node [swap] {$\eta$} (D2);
\draw[->, bend left=50] (B2) to node  {$\sigma - \gamma h$} (D2);
\end{tikzpicture}
\end{center}
In words, we may remove $\beta$ from the right-hand side. However, in doing
so, we perturb the differential by arrows which factor through $\alpha$ or
$\gamma$. Note that, if $\Morph^*(B,C) \simeq 0$ then every cycle $f \in \Morph^*(B,C)$ is a boundary.
\end{lemma}

\begin{proof}
Since $\beta$ is a boundary there exists a homotopy $h : B \to C$ such that $\beta = dh - hd$. This allows us to define maps $\varphi = \varphi^1, \varphi^{-1}$  where
$$\varphi^{\pm 1} = \left(\begin{array}{cccc} \Id_A & & & \\  & \Id_B & & \\ & \pm h & \Id_C & \\  & & & \Id_D \end{array}\right).$$
Notice that $\varphi \varphi^{-1} = \Id$ and $\varphi^{-1} \varphi =
\Id$. If $d$ is the differential on the left-hand side then $\varphi d_{\Tot(T)} \varphi^{-1}$ is the differential on the right-hand side.
\end{proof}

\section{Computing spaces of maps using duality}\label{homspacessec}
In this section we recall a duality for $\Hom$-spaces inside of the category
$\Kom(n,m)$. For each sequence $\e\in\aL_n$, chain level analogues
$Q_\e\in\Kom(n)$ of the elements $q_\e\in\TL_n$ found in Section
\ref{genjwprojsec} are introduced. In Theorem \ref{homComputation} the
duality statement is used to prove that $\Hom$-spaces between convolutions
of $Q_\e$ and $Q_\nu$ respect the dominance order.

\subsection{Duality}\label{dualitysubsec}
\begin{notation}
Denote by $\Kom(n)^b$ the subcategory of $\Kom(n)$ consisting of chain
complexes which are bounded on both sides in homological degree.
\end{notation}

\begin{definition}{$(C^\vee)$}\label{reflectdef}
  If $C \in \Kom(n)^b$ then the {\em dual} complex $C^\vee$ is obtained by
  reflecting all of the diagrams in the chain complex about the $x$-axis and
  reversing both the quantum and homological gradings.
\end{definition}

\begin{remark}
When defining the invariants of tangles which live in
$\Kom(n)^b$ (see \cite{DBN,KH}) the chain complex associated to a negative
crossing can be obtained from the chain complex associated to a positive
crossing by applying this functor.
\end{remark}

\begin{remark}
One can show that $(C^{\vee})^{\vee} \cong C$ and that $-^\vee :
\Kom(n)^b\to\Kom(n)^b$ preserves homotopy. 
\end{remark}

Our primary interest in $-^\vee$ stems from its behavior with respect to the
pairing:
$$\Kom(n)^b\times\Kom(n)^b \to \Kom(0)^b \conj{where} (A,B) \mapsto \Hom^*(A,B).$$
The computation of $\Hom$-spaces in $\Kom(n)^b$ can be simplified using the planar algebra trick which is illustrated below. 
$$\Hom^*\left(\CPic{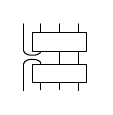},\CPic{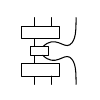}\right) \cong \CPic{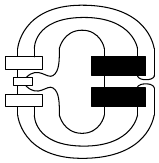}$$

The boxes above represent choices of chain complexes in $\Kom(n)^b$ so that
on the left-hand side of this equation is the chain complex of maps between
the two objects in the differential graded category $\Kom^*(n)^b$.  On the
right-hand side of this equation is the chain complex in $\Kom(0)^b$ formed
by dualizing the first of the two objects and then connecting its free end
points to those of the second object. This is identified with the chain
complex of abelian groups on the left after applying the functor
$\Hom(\varnothing, -)$.

The theorem below contains a precise statement of the general case.

\begin{theorem}\label{dualthm}
For all $A,B \in \Kom(n)^b$, the Markov trace of the chain complex $B\otimes
A^\vee \in \Kom(n)^b$ computes the extended $\Hom$-space from $A$ to $B$.
$$\Hom^*(A,B) \cong \tr(B\otimes A^\vee)$$
\end{theorem}
\begin{proof}
If $f : D\to D'$ is a map between diagrams $D$ and $D'$ in $\Cob(n)$ then
there is a canonical way to construct an element of $D'^\vee \otimes
D$. This commutes with convolutions and so respects the differential.
\end{proof}

\begin{remark}
  This duality has been explored in \cite[Thm 1.3]{CMW}, in \cite{R2} this
  duality was denoted by $-^\Diamond$ and in \cite{Cautis} it is denoted by
  $\mathbb{D}(-)$. Moreover, in each rigid monoidal category there is an
  isomorphism $\Hom(1, A^\vee\otimes B) \to \Hom(A,B)$ \cite{BaK}. In our
  setting we can identify the left-hand side with the chain complex
  determined by the Markov trace.
\end{remark}

Since the category $\Kom(n)^b$ of bounded complexes is closed under the
duality $-^\vee$ it has an extra bit of symmetry which is lacking in the
category $\Kom(n)$. One way to use the theorem above in our context is to
allow one term in the $\Hom$-pairing to be a chain complex in $\Kom(n)$ and
require the other term to be a chain complex in $\Kom(n)^b$:
$$\Kom(n)^b\times \Kom(n) \to \Kom(0)\conj{or}\Kom(n)\times \Kom(n)^b \to \Kom(0).$$
Then Theorem \ref{dualthm} above continues to hold. This is all that is
necessary for the proof of Theorem \ref{homComputation} below.  For an
alternative viewpoint, see \cite{H}.

\subsection{Categorical quasi-idempotents}\label{catquasisubsec}
In this section we associate to each sequence $\e\in\aL_n$ a special chain
complex $Q_\e\in\Kom(n)$. This construction is directly analogous to the
definition of $q_\e$ in Section \ref{genjwprojsec}. 

\begin{definition}\label{bigqedef}
If $\e \in \aL_n$ and $|\e| = k$ then there is an element  $T_\e \in \Kom(k,n)$ which is defined inductively by $T_{(1)} = 1$, 
$$T_{\e\cdot (+1)} = \CPic{aepsilon1} \quad\normaltext{ and }\quad T_{\e\cdot (-1)} = \CPic{aepsilon2}$$
where the box represents a universal projector $P_k$ (Theorem
\ref{uprojectorthm}) and the marshmallow-shaped
region represents the element $T_\e$. The special element $Q_\e\in\Kom(n)$ is equal to the top $T_\e$ composed with its reverse,
$$Q_\e = T_\e\otimes \bar{T}_\e.$$
\end{definition}

In other words, replacing $p_k$ with $P_k$ in the definition of $q_\e$
gives us $Q_\e$. The graded Euler characteristic of $Q_\e$ is the element of
$\TL_n$ obtained from $q_\e$ after identifying its coefficients with
elements of $\ZZ[[q]][q^{-1}]$. 

The chain complexes $Q_\e$ will be used extensively in Sections
\ref{explicitsec} and \ref{mainconstructionsec}.

\subsection{Hulls of $Q_\e$ are perpendicular}\label{hullsareperpsec}
Before stating the main theorem in this section, we must introduce a lemma
which will be used in its proof.

\begin{lemma}\label{lemlemlem}
If $N\in\Kom(n)$ then the collection of complexes annihilated by $N$,
$$\Ann(N) = \{ M \in \Kom(n) : M\otimes N \simeq 0\}$$
is closed under convolution. 
\end{lemma}
\begin{proof}
If $\{(E_i), q_{ij}\}$ is a twisted complex with $E_i \in \Ann(N)$ then by Lemma \ref{contractLemma}
\[
\Tot(\{(E_i),q_{ij}\})\otimes N \cong \Tot(\{(E_i\otimes N), q_{ij}\otimes \Id_N\})\simeq 0.
\]
\end{proof}

\begin{remark}
In particular, if $Q_\e \otimes N \simeq 0$ then $\inp{Q_\e}\subset \Ann(N)$ and so each element $[Q_\e]\in\inp{Q_\e}$ in the hull of $Q_\e$ satisfies $[Q_\e]\otimes N\simeq 0$.
\end{remark}

The following theorem tells us that each convolution of $Q_\e$ is
perpendicular to each convolution of $Q_\nu$ with respect to the
$\Hom$-pairing when $\e\ndominatedBy\nu$. This is used in Theorem
\ref{conethingthm}, in conjunction with Proposition \ref{obstprop}, to
inductively construct the projectors $P_\e$.

\begin{theorem}\label{homComputation}
If $\e,\nu\in\aL_n$ are sequences and $\e\ndominatedBy\nu$ then
$$\Hom^*([Q_\e],[Q_\nu])\simeq 0$$ 
for each $[Q_\e]\in\inp{Q_\e}$ and $[Q_\nu]\in\inp{Q_\nu}$.
\end{theorem}
\begin{proof}
Suppose that $n\in\ZZ_+$ and $N = [Q_\nu]^n$ is the $n$th chain group of the
chain complex $[Q_\nu]$. The condition $\e\ndominatedBy\nu$ implies that there is an $i$ such that 
$$\e_1+\dots +\e_i>\nu_1+\dots +\nu_i.$$
Let $k = \e_1+\dots +\e_i$ and $l = \nu_1+\dots +\nu_i$. By definition of $Q_\nu$, every summand $a$ of $N = [Q_\nu]^n$ can be written as
$$a = c\otimes (b\sqcup 1_{n-i}) \conj{for some} b\in \TL(i,l)$$
and $b$ satisfies $Q_\e\otimes (b\sqcup 1_{n-i})^\vee\simeq 0$ (for the same reasons as Proposition \ref{orthogonality}). Hence, $Q_\e\otimes N^\vee \simeq 0$. Lemma \ref{lemlemlem} and Theorem \ref{dualthm} imply that
\[
\Hom^*([Q_\e],[Q_\nu]^n) = \Hom^*([Q_\e],N)=  \Hom^*([Q_\e]\otimes N^\vee, 1_n) \simeq 0.
\]

Finally, we observe that $\Hom^*([Q_\e],[Q_\nu]) = \Tot^\Pi(E)$ (see the Remark in Section \ref{totremark}) where
\[
E = \Hom^*([Q_\e],[Q_\nu]^0)\rightarrow \Hom^*([Q_\e],[Q_\nu]^1)\rightarrow  \dots .
\]
So $\Hom^*([Q_\e],[Q_\nu])$ is a convolution of contractible chain complexes
and Lemma \ref{contractLemma} implies the theorem.
\end{proof}

\begin{remark}
  The above considerations admit marginal generalizations which are not
  necessary for the proof of the main theorem. For instance, suppose that
  $A\otimes^{\Pi}B$ denotes the tensor product: $(A\otimes^\Pi B)_n =
  \prod_{i+j=n}A_i\otimes B_j$.  Then the argument above shows that
  $[Q_\nu]\otimes^{\Pi} [Q_\e]\simeq 0$ unless $\e\unlhd \nu$ and the
  following adjunction holds:
$$\Hom^\ast(C\otimes A,B)\cong \Hom^\ast(C,B\otimes^\Pi A^\vee).$$ 
These observations lead to the stronger statement below.
\begin{statement} Suppose that $A,B\in \Kom(n)$, $[Q_\e]\in\inp{Q_\e}$ and $[Q_\nu]\in\inp{Q_\nu}$.  Then $\e\ndominatedBy \nu$ implies that $\Hom^\ast(A\otimes [Q_\e], B\otimes [Q_\nu])$ is contractible.\end{statement}

\end{remark}

\section{Explicit constructions of resolutions of identity}\label{explicitsec}
In this section the higher order projectors are constructed for $n = 2,3,
\normaltext{ and } 4$. The general construction can be found in Section
\ref{mainconstructionsec}. Many of the important features of this proof can
be seen concretely here when $n=4$.  The subscripts used in this section
correspond to the sequences introduced in Definition \ref{domOrderstuff}.

\subsection{Two strands: $P_{(1,-1)}$ and $P_{(1,1)}$}\label{twostrandssec}
The second projector $P_2$ can be represented by chain complex of the form
\begin{center}
\begin{tikzpicture}[scale=10]
\matrix (m) [matrix of math nodes, row sep=3em,
column sep=3em, text height=1.5ex, text depth=0.25ex]
{ \CPic{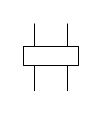}  =   \CPic{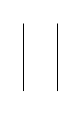}  & tq \CPic{n2-s}  & t^2 q^3 \CPic{n2-s} & t^3 q^5 \CPic{n2-s} \cdots\\ };
\path[->]
(m-1-1) edge (m-1-2)
(m-1-2) edge (m-1-3)
(m-1-3) edge (m-1-4);
\end{tikzpicture}
\end{center}
where the first map is a saddle and the last two maps alternate between a
difference and a sum of two dots, see \cite{CK}. Recall the box notation from Definition \ref{boxnotation}, we write
\begin{center}
\begin{tikzpicture}[scale=10]
\matrix (m) [matrix of math nodes, row sep=3em,
column sep=3em, text height=1.5ex, text depth=0.25ex]
{ \CPic{p2box} = \CPic{n2-1}  & t \CPicTailbox{n2-s} \quad\normaltext{ and }\quad  \CPic{n2-1} \simeq \CPicTailbox{n2-s} & \CPic{p2box}  \\ };
\path[->]
(m-1-1) edge (m-1-2)
(m-1-2) edge (m-1-3);
\end{tikzpicture}
\end{center}
where the map defining the first cone is the saddle appearing in the
definition of $P_2$ and the map in the second cone is the inclusion of the
tail into $P_2$. Let us write $t P_{(1,-1)}$ for the subcomplex of $P_2$
consisting of terms in homological degree greater than zero and set
$P_{(1,1)} = P_2$. There is a map $i : P_{(1,-1)} \to P_{(1,1)}$, with
$\deg_t(i) = 1$, which satisfies $\Cone(i)\simeq 1_2$ where $1_2$ is the
identity diagram illustrated above.

The chain complex $P_{(1,-1)}$ is idempotent and the map $i$ gives the
resolution of identity.

\subsection{Three strands: $P_{(1,-1,1)}$, $P_{(1,1,-1)}$ and $P_{(1,1,1)}$}\label{threestrandssec}
The identity object $1_3$ on three strands is given by the union of the
identity object on two strands together with an extra strand, $1_3 = 1_2
\sqcup 1_1$. Applying $-\sqcup 1$ to the resolution of identity in the
previous section we obtain:
\begin{center}
\begin{tikzpicture}[scale=10]
\matrix (m) [matrix of math nodes, row sep=3em,
column sep=3em, text height=1.5ex, text depth=0.25ex]
{ \CPic{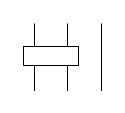}\;\; = \CPic{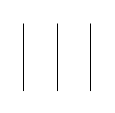}  & t \CPicTailboxx{n3-e1} \quad\normaltext{ and }\quad  \CPic{n3-1} \simeq \CPicTailboxx{n3-e1} & \CPic{n3-p2box}\quad.  \\ };
\path[->]
(m-1-1) edge (m-1-2)
(m-1-2) edge (m-1-3);
\end{tikzpicture}
\end{center}

Lemma \ref{standardformlemma} implies that the third universal projector
$P_3 = P_{(1,1,1)}$ can be chosen to be equal to the cone $P_{(1,1)}\sqcup 1 \to t P_{(1,1,-1)}$. Pictorially,
\begin{center}
\begin{tikzpicture}[scale=10]
\matrix (m) [matrix of math nodes, row sep=3em,
column sep=3em, text height=1.5ex, text depth=0.25ex]
{ \CPic{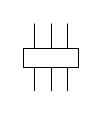} = \CPic{n3-p2box}\quad & t \CPicTailboxx{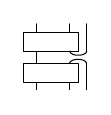}   & \coneqnOne\\ };
\path[->]
(m-1-1) edge (m-1-2);
\end{tikzpicture}
\end{center}

Consider the contractible chain complex $\Cone(-\Id) = P_{(1,1,-1)} \to t
P_{(1,1,-1)}$. Using the second equation above and from gluing on the
contractible chain complex it follows that $1_3$ is homotopic to
\begin{center}
\begin{tikzpicture}[scale=10, node distance=2.5cm]
\node (D) {$\CPicTailboxx{n3-e1}$};
\node (F) [right of=D] {$\CPic{n3-p2box}$};
\node (E) [below of=D] {$\CPicTailboxx{pbox-ppm}$};
\node (G) [right of=E] {$t\CPicTailboxx{pbox-ppm}$};
\draw[->] (D) to node {} (E);
\draw[->] (D) to node {} (F);
\draw[->] (E) to node {} (G);
\draw[->] (F) to node {} (G);
\end{tikzpicture}
\end{center}
by Lemma \ref{addcontractiblelemma}. Using the triangle \coneqnOne\ above
and reassociating allows us to write this complex in terms of the projectors
$P_{(1,-1,1)} = P_{(1,-1)}\sqcup 1$, $P_{(1,1,-1)}$ and $P_{(1,1,1)}$. The
identity object $1_3$ is homotopy equivalent to $R_3$.
\begin{center}
\begin{tikzpicture}[scale=10, node distance=2.5cm]
\node (D) {$\CPicTailboxx{n3-e1}$};
\node (F) [right of=D] {$\CPic{p3box}$};
\node (E) [below of=D] {$\CPicTailboxx{pbox-ppm}$};
\draw[->] (D) to node {} (E);
\draw[->] (D) to node {} (F);
\draw[->, bend right] (E) to node {} (F);
\node (R) [right=.5cm of F] {$=$};
\node (X) [right=.5cm of R] {$P_{(1,-1,1)}$};
\node (Y) [right of=X] {$P_{(1,1,1)}$};
\node (Z) [below of=X] {$P_{(1,1,-1)}$};
\draw[->] (X) to node {} (Z);
\draw[->] (X) to node {} (Y);
\draw[->, bend right] (Z) to node {} (Y);
\end{tikzpicture}
\end{center}
The maps above are compositions of inclusions of tails and differentials
from chain complexes of projectors.

\subsection{Four strands}\label{fourstrandssec}
In the previous section we obtained a resolution of the identity on three
strands.  The identity object $1_4$ on four strands is given by the union of
the identity object on three strands together with an extra strand, $1_4 = 1_3
\sqcup 1_1$. Applying $-\sqcup 1$ to the resolution of identity in the
previous section we obtain the diagram pictured below.

\begin{center}
\begin{tikzpicture}[scale=10, node distance=2.5cm]
\node (S) {$\CPic{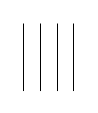}$};
\node (T) [right=.5cm of S] {$\simeq$};
\node (D) [right=.5cm of T] {$\CPicTailboxx{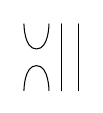}$};
\node (F) [right of=D] {$\CPic{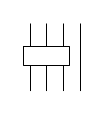}$};
\node (E) [below of=D] {$\CPicTailboxx{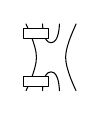}$};
\draw[->] (D) to node {} (E);
\draw[->] (D) to node {} (F);
\draw[->, bend right] (E) to node {} (F);
\end{tikzpicture}
\end{center}

Now Lemma \ref{standardformlemma} implies that the fourth universal
projector $P_4 = P_{(1,1,1,1)}$ can be chosen to be equal to the cone
$P_3 \sqcup 1 \to t P_{(1,1,1,-1)}$.

\begin{center}
\begin{tikzpicture}[scale=10]
\matrix (m) [matrix of math nodes, row sep=3em,
column sep=3em, text height=1.5ex, text depth=0.25ex]
{ \CPic{p4box} = \CPic{n3sqcup1} & t \CPicTailboxx{p4tail} \\ };
\path[->]
(m-1-1) edge (m-1-2);
\end{tikzpicture}
\end{center}

Using Lemma \ref{addcontractiblelemma} we can add the contractible chain
complex $\Cone(-\Id) = P_{(1,1,1,-1)} \to t P_{(1,1,1,-1)}$ to the
decomposition above to obtain a homotopy equivalent complex on the left-hand
side below.

\begin{center}
\begin{tikzpicture}[scale=10, node distance=2.5cm]
\node (D) {$\CPicTailboxx{n4-e1}$};
\node (F) [right of=D] {$\CPic{n3sqcup1}$};
\node (E) [below of=D] {$\CPicTailboxx{n4-unp2}$};
\node (H) [below of=F] {};
\node (J) [right of=H] {$t \CPicTailboxx{p4tail}$};
\node (I) [below of=H] {$\CPicTailboxx{p4tail}$};
\draw[->] (D) to node {} (E);
\draw[->] (D) to node {} (F);
\draw[->, bend right] (E) to node {} (F);
\draw[->, bend right] (I)  to node [swap] {$-\Id$} (J);
\draw[->, bend left] (F) to node {} (J);
\draw[->, bend right] (E) to node {} (I);
\draw[->, bend left] (D) to node {} (I);
\node (P1) [right=.5cm of J] {$\cong$};
\node (P2) [right of=P1] {};
\node (D2) [above of=P2] {$\CPicTailboxx{n4-e1}$};
\node (F2) [right of=D2] {$\CPic{p4box}$};
\node (E2) [below of=D2] {$\CPicTailboxx{n4-unp2}$};
\node (I2) [below of=E2] {$\CPicTailboxx{p4tail}$};
\draw[->] (D2) to node {} (E2);
\draw[->] (D2) to node {} (F2);
\draw[->, bend right] (E2) to node {} (F2);
\draw[->] (E2) to node {} (I2);
\draw[->, bend right=50] (D2) to node {} (I2);
\draw[->, bend right=50] (I2) to node {} (F2);
\end{tikzpicture}
\end{center}

Reassociating allows us to replace $P_{(1,1,1)}\sqcup 1$ in the resolution
of identity and yields the isomorphic complex containing the projector $P_4$
on the right-hand side above. Unfortunately, we aren't done because our
resolution of identity still consists of terms which do not factor through
universal projectors. In order to replace the two offending terms,
$P_{(1,1,-1)}\sqcup 1$ and $P_{(1,-1,1)}\sqcup 1$, a bit of work
remains. The process by which we replace $P_{(1,1,-1)}\sqcup 1$ will
illustrate the general strategy.

We can construct the following chain complex,
\tikzset{node distance=2cm, auto}
\begin{center}
\begin{tikzpicture}[scale=10]
\matrix (m) [matrix of math nodes, row sep=3em,
column sep=3em, text height=3.5ex, text depth=3.5ex]
{ \MPic{n4-e1} & t\MPic{n4-e1}  & t^2 \MPic{n4-e1}  & \cdots  &  \\
 t \MPic{n4-e1e3} & t^2\MPic{n4-e1e3}  & t^3 \MPic{n4-e1e3}  & \cdots  &  \\
 t^2 \MPic{n4-e1e3} & t^3\MPic{n4-e1e3}  & t^4 \MPic{n4-e1e3}  & \cdots  &  \\
 \vdots & \vdots  & \vdots  &   &  \\ };
\path[->]
(m-1-1) edge (m-1-2)
(m-1-2) edge (m-1-3)
(m-1-3) edge (m-1-4);
\path[->]
(m-2-1) edge (m-2-2)
(m-2-2) edge (m-2-3)
(m-2-3) edge (m-2-4);
\path[->]
(m-3-1) edge (m-3-2)
(m-3-2) edge (m-3-3)
(m-3-3) edge (m-3-4);
\path[->]
(m-1-1) edge (m-2-1)
(m-2-1) edge (m-3-1)
(m-3-1) edge (m-4-1);
\path[->]
(m-1-2) edge (m-2-2)
(m-2-2) edge (m-3-2)
(m-3-2) edge (m-4-2);
\path[->]
(m-1-3) edge (m-2-3)
(m-2-3) edge (m-3-3)
(m-3-3) edge (m-4-3);
\end{tikzpicture}
\end{center}
In the more concise bracket notation, this chain complex is $P_{(1,-1,1,1)} =
\MPicTailboxx{n4-e1p2}$. Notice the top row of this bicomplex is
$P_{(1,-1,1)}\sqcup 1$; in bracket notation, this is $\MPicTailboxx{n4-e1}$.
The columns of the bicomplex $P_{(1,1,-1,1)}$ are given by $\MPic{n4-e1p2}$.

We define $P_{(1,-1,1,-1)}$ to be the tail of the bicomplex
$P_{(1,1,-1,1)}$: the subcomplex consisting of all rows beyond the first
(shifted down by 1). In bracket notation, $P_{(1,-1,1,-1)} =
\MPicTailboxx{n4-e1e3}.$ The vertical differential of the bicomplex
$P_{(1,1,-1,1)}$ determines a map $\delta : P_{(1,-1,1)}\sqcup 1 \to
P_{(1,-1,1,-1)}$ such that $P_{(1,-1,1,1)} = \Cone(\delta).$ In pictures,
\begin{center}
\begin{tikzpicture}[scale=10]
\matrix (m) [matrix of math nodes, row sep=3em,
column sep=3em, text height=1.5ex, text depth=0.25ex]
{ \CPicTailboxx{n4-e1p2} = \CPicTailboxx{n4-e1}  & t \CPicTailboxx{n4-e1e3} .   \\ };
\path[->]
(m-1-1) edge (m-1-2);
\end{tikzpicture}
\end{center}

Using Lemma \ref{addcontractiblelemma} we can add the contractible chain
complex $\Cone(-\Id) = P_{(1,-1,1,-1)} \to t P_{(1,-1,1,-1)}$ to our
decomposition above and reassociate. The identity object $1_4$ is homotopy
equivalent to the left complex pictured below. The complex on the right is
obtained by reassociating.
\begin{center}
\begin{tikzpicture}[scale=7, node distance=2.5cm]
\node (E) {$P_{(1,-1,1,-1)}$};
\node (F) [right=.01cm of E] {};
\node (G) [below of=E] {$t P_{(1,-1,1,-1)}$};
\node (H) [right=.01cm of F] {$P_{(1,-1,1)} \sqcup 1$};
\node (P1) [below of=H] {};
\node (I) [below of=P1] {$P_{(1,1,1,-1)}$};
\node (J) [below of=H] {$P_{(1,1,-1)}\sqcup 1$};
\node (K) [right of=H] {$P_{(1,1,1,1)}$};
\draw[->] (E) to node [swap] {$-\Id$} (G);
\draw[->, bend right] (H) to node {} (G);
\draw[->, bend right=50] (H) to node {} (I);
\draw[->] (J) to node {} (I);
\draw[->, bend right] (J) to node {} (K);
\draw[->, bend right=50] (I) to node {} (K);
\draw[->] (H) to node {} (K);
\draw[->] (H) to node {} (J);
\node (Z1) [right=.25cm of K] {};
\node (W1) [below of=K] {};
\node (W2) [right=.75cm of W1] {$\cong$};
\node (A2) [right=.25cm of Z1] {$P_{(1,-1,1,-1)}$};
\node (B2) [right of=A2] {$P_{(1,-1,1,1)}$};
\node (C2) [right of=B2] {$P_{(1,1,1,1)}$};
\node (D2) [below of=B2] {$P_{(1,1,-1)}\sqcup 1$};
\node (E2) [below of=D2] {$P_{(1,1,1,-1)}$};
\draw[->] (A2) to node {} (B2);
\draw[->] (B2) to node {} (C2);
\draw[->] (B2) to node {} (D2);
\draw[->] (D2) to node {} (E2);
\draw[->, bend right=50] (B2) to node {} (E2);
\draw[->, bend right] (D2) to node {} (C2);
\draw[->, bend right=50] (E2) to node {} (C2);
\end{tikzpicture}
\end{center}

We still have to replace the subcomplex $P_{(1,1,-1)}\sqcup 1$ with a
complex that factors through $P_2$. In order to accomplish this task we
construct a chain complex $P_{(1,1,-1,-1)} = \MPicTailboxx{n4-pp0}$ and a
chain map $\gamma : P_{(1,1,-1)}\sqcup 1 \to P_{(1,1,-1,-1)}$ so that $
P_{(1,1,-1,1)} = \Cone(\gamma) = \MPicTailboxx{n4-pp2}.$ The complex
$P_{(1,1,-1)} \sqcup 1 = \MPicTailboxx{n4-unp2}$ can be written as an
iterated cone
$$\CPicTailboxx{n4-unp2} = \Bigg(\Bigg( \Bigg(\Bigg(\CPic{n4-unp2} \to t \CPic{n4-unp2}\Bigg) \to t^2 \CPic{n4-unp2}\Bigg) \to t^3 \CPic{n4-unp2}\Bigg)\to \cdots \Bigg).$$
We can also write
\begin{center}
\begin{tikzpicture}[scale=7]
\matrix (m) [matrix of math nodes, row sep=3em,
column sep=3em, text height=1.5ex, text depth=0.25ex]
{ \CPic{n4-pp2} = \CPic{n4-unp2}  & t \CPicTailboxx{n4-pp0} .   \\ };
\path[->]
(m-1-1) edge (m-1-2);
\end{tikzpicture}
\end{center}
(In the formula above, a degree shift of $[2]/[3]$ has been omitted,
following the convention in Definition \ref{boxnotation}.) Using this map we
now construct a new triangle of the form
\begin{center}
\begin{tikzpicture}[scale=7]
\matrix (m) [matrix of math nodes, row sep=3em,
column sep=3em, text height=1.5ex, text depth=0.25ex]
{ \CPicTailboxx{n4-pp2} = \CPicTailboxx{n4-unp2}  & t \CPicTailboxxTwo{n4-pp0} .   \\ };
\path[->]
(m-1-1) edge (m-1-2);
\end{tikzpicture}
\end{center}

This process is also carried out in Theorem \ref{pespexistencethm}. We use
double brackets above to emphasize that the term on the left is a
convolution of convolutions and also to distinguish it from the complex
$\MPicTailboxx{n4-pp0}$ which is the tail of \MPic{n4-pp2}.

The first step is to form the cone on the first term of $P_{(1,1,-1)} \sqcup 1$.
\begin{center}
\begin{tikzpicture}[scale=7, node distance=2.5cm]
\node (A) {$\CPic{n4-unp2}$};
\node (B) [right of=A] {$\CPic{n4-unp2}$};
\node (C) [right of=B] {$\cdots$};
\node (D) [below of=A] {$\CPicTailboxx{n4-pp0}$};
\draw[->] (A) to node {} (B);
\draw[->] (B) to node {} (C);
\draw[->] (A) to node {} (D);
\end{tikzpicture}
\end{center}
Reassociating shows that the first term in this complex agrees with the
desired complex. Now assume by induction that we can form a chain complex in
which the first $N$ terms of $P_{(1,1,-1)} \sqcup 1$ have been written in
this way.

\begin{center}
\begin{tikzpicture}[scale=7, node distance=2.5cm]
\node (A) {$\CPic{n4-unp2}$};
\node (A1) [right of=A] {$\CPic{n4-unp2}$};
\node (A2) [right of=A1] {$\CPic{n4-unp2}$};
\node (B) [right of=A2] {$\CPic{n4-unp2}$};
\node (C) [right of=B] {$\cdots$};
\node (D) [below of=A] {$\CPicTailboxx{n4-pp0}$};
\node (E) [right of=D] {$\CPicTailboxx{n4-pp0}$};
\node (F) [right of=E] {$\CPicTailboxx{n4-pp0}$};
\draw[->] (A) to node {} (A1);
\draw[->] (A1) to node {} (A2);
\draw[->] (A2) to node {} (B);
\draw[->] (B) to node {} (C);
\draw[->] (A) to node {} (D);
\draw[->] (A1) to node {} (E);
\draw[->] (A2) to node {} (F);
\draw[->] (D) to node {} (E);
\draw[->] (E) to node {} (F);
\draw[->] (A) to node {} (E);
\draw[->] (A) to node {} (F);
\draw[->] (A1) to node {} (F);
\end{tikzpicture}
\end{center}
We draw the diagonal arrows to emphasize that the maps in this contruction
necessarily propagate in a non-trivial way.

After grouping the first $N$ terms of the top and bottom rows within
parenthesis we consider taking the cone on the $N+1$st term
\begin{center}
\begin{tikzpicture}[scale=7, node distance=3cm]
\node (A) {$\Bigg(\CPic{n4-pp2}\Bigg)$};
\node (B) [right of=A] {$\CPic{n4-unp2}$};
\node (C) [right of=B] {$\cdots$};
\node (D) [below of=A] {};
\node (E) [below of=B] {$\CPicTailboxx{n4-pp0}$};
\draw[->] (A) to node {$\alpha$} (B);
\draw[->] (B) to node {} (C);
\draw[->] (B) to node {$\delta$} (E);
\draw[->, bend right] (A) to node {$-h$} (E);
\end{tikzpicture}
\end{center}
After taking shifts into account, the composition $\delta\circ\alpha$ is a chain map of degree $0$ and
$$\delta\circ\alpha \in \Morph^*\left(\Bigg(\CPic{n4-pp2}\Bigg), \CPicTailboxx{n4-pp0}\right) \simeq 0.$$
The $\Morph$-complex is contractible by Theorem \ref{homComputation}. Proposition \ref{obstprop} allows us to produce a chain
complex with $N+1$ terms of the desired form. This process is stable, adding
the $N+1$st map does not change any maps which appear earlier, because in
Proposition \ref{obstprop}, the map $\gamma$ is an extension of the map
$\beta$. Since there are countably many terms we can use this process to
produce the chain complex, $P_{(1,1,-1,1)}$, that we want.

By construction the top row is $P_{(1,1,-1)}\sqcup 1$ and so we define the bottom row to be $t P_{(1,1,-1,-1)}$. The non-horizontal components of the differential yield a chain map 
$$\eta : P_{(1,1,-1)}\sqcup 1 \to P_{(1,1,-1,-1)}$$ 
such that $P_{(1,1,-1,1)} = \Cone(\eta)$. Our program is resumed by
replacing the $P_{(1,1,-1)}\sqcup 1$ term above.  By introducing the
contractible term $\Cone(-\Id_{P_{(1,1,-1,-1)}}) = P_{(1,1,-1,-1)}\to t
P_{(1,1,-1,-1)}$ to our last complex, we see that $1_4$ is homotopy
equivalent to the diagram pictured below.

\begin{center}
\begin{tikzpicture}[scale=10, node distance=2.5cm]
\node (L) {};
\node (A) [right=.5cm of L] {$P_{(1,-1,1,-1)}$};
\node (B) [right of=A] {$P_{(1,-1,1,1)}$};
\node (C) [right of=B] {$P_{(1,1,1,1)}$};
\node (D) [below of=B] {$P_{(1,1,-1)}\sqcup 1$};
\node (E) [below of=D] {$P_{(1,1,1,-1)}$};
\node (N) [below of=L] {$P_{(1,1,-1,-1)}$};
\node (Z) [below of=A] {};
\node (M) [below of=N] {$tP_{(1,1,-1,-1)}$};
\draw[->, bend left=5] (B) to node {} (N);
\draw[->, bend left=5] (D) to node [swap] {$\eta$} (M);
\draw[->] (N) to node [swap] {$-\Id$} (M);
\draw[->] (A) to node {} (B);
\draw[->] (B) to node {} (C);
\draw[->] (B) to node {} (D);
\draw[->] (D) to node {} (E);
\draw[->, bend right=50] (B) to node {} (E);
\draw[->, bend right] (D) to node {} (C);
\draw[->, bend right=50] (E) to node {} (C);
\end{tikzpicture}
\end{center}

We conclude by reassociating and using the Combing Lemma
\ref{combinTheHairs} to exchange the bad arrow
$P_{(1,-1,1,1)} \to P_{(1,1,-1,-1)}$ with an arrow $P_{(1,-1,1,-1)} \to
P_{(1,1,-1,-1)}$ that respects the dominance order $\dominatedBy$ on $\aL_4$
(Definition \ref{domOrderstuff}). The object $1_4$ is homotopic to the complexes
pictured below.

\begin{center}
\begin{tikzpicture}[scale=10, node distance=2.5cm]
\node (A) {$P_{(1,-1,1,-1)}$};
\node (B) [right of=A] {$P_{(1,-1,1,1)}$};
\node (C) [right of=B] {$P_{(1,1,1,1)}$};
\node (D) [below of=B] {$P_{(1,1,-1,1)}$};
\node (E) [below of=D] {$P_{(1,1,1,-1)}$};
\node (N) [below of=A] {$P_{(1,1,-1,-1)}$};
\draw[->, bend right] (B) to node {} (N);
\draw[->] (N) to node {} (D);
\draw[->] (A) to node {} (B);
\draw[->] (B) to node {} (C);
\draw[->] (B) to node {} (D);
\draw[->] (D) to node {} (E);
\draw[->, bend right=45] (B) to node {} (E);
\draw[->, bend right] (D) to node {} (C);
\draw[->, bend right=50] (E) to node {} (C);
\node (Z1) [below of=C] {};
\node (Z2) [right=1cm of Z1] {$\cong$};
\node (N2) [right=.5cm of Z2] {$P_{(1,1,-1,-1)}$};
\node (A2) [above of=N2] {$P_{(1,-1,1,-1)}$};
\node (B2) [right of=A2] {$P_{(1,-1,1,1)}$};
\node (C2) [right of=B2] {$P_{(1,1,1,1)}$};
\node (D2) [below of=B2] {$P_{(1,1,-1,1)}$};
\node (E2) [below of=D2] {$P_{(1,1,1,-1)}$};
\draw[->] (A2) to node {} (N2);
\draw[->] (N2) to node {} (D2);
\draw[->] (A2) to node {} (B2);
\draw[->] (B2) to node {} (C2);
\draw[->] (B2) to node {} (D2);
\draw[->] (D2) to node {} (E2);
\draw[->, bend right=45] (B2) to node {} (E2);
\draw[->, bend right] (D2) to node {} (C2);
\draw[->, bend right=50] (E2) to node {} (C2);
\end{tikzpicture}
\end{center}

The end result is a resolution of identity on four strands in which all of
the terms factor through universal projectors of the form $P_{4-2k}$ for
$k=0,1,2$ and all maps between terms respect the dominance order.

\begin{vista}
  In order to accomplish our task, we needed two basic manuevers. The first
  was gluing a contractible chain complex onto our resolution without
  changing the homotopy type using Lemma \ref{addcontractiblelemma}. The
  second was the construction of chain complexes suitable for
  substitution. The first step is provided by Lemma \ref{standardformlemma}.
  While Proposition \ref{obstprop} and Theorem \ref{homComputation} allow us
  to construct more sophisticated sorts of substitutions.
\end{vista}

In order to construct the $P_\e$, a general version of the argument given
above is carried out in Section \ref{mainconstructionsec}. The reader may
refer to this section for intuition.

\section{General construction of the resolution of identity}\label{mainconstructionsec}
In this section we categorify the equations
$$1_n = \sum_{\e \in \aL_n} p_\e \conj{ and } p_\e p_{\nu} = \d_{\e \nu} p_\e$$
of Theorem \ref{decompositionOfIdentity}. This is accomplished by
constructing chain complexes $P_\e\in \Kom(n)$ for each sequence $\e \in
\aL_n$ which satisfy idempotence and orthogonality properties, $P_\e\otimes
P_\nu \simeq \delta_{\e\nu} P_\e$.  In the process of constructing the
projectors $P_\e$, we build the {\em resolution of identity} $R_n$; a
convolution of projectors $P_\e$ which satisfies $1_n \simeq R_n$.  The
Euler characteristic $\K_0(1_n) = \K_0(R_n)$ can be identified with the
first equation above.

In Theorem \ref{homComputation}, we showed that $\e \ndominatedBy \nu$
implies $\Hom^*([Q_{\e}],[Q_{\nu}]) \simeq 0$, for each two convolutions
$[Q_\e]\in\inp{Q_\e}$ and $[Q_{\nu}]\in\inp{Q_{\nu}}$.  The theorem below
exploits this fact in order to build triangles relating convolutions in the hulls 
$\inp{Q_{\e\cdot (+1)}}$, $\inp{Q_{\e\cdot (-1)}}$ and $\inp{Q_{\e}\sqcup
  1}$.  An immediate consequence is Corollary \ref{pespexistenceprop}, which
constructs chain complexes $P_\e \in\inp{Q_\e}$. These chain complexes categorify the idempotents $p_\e\in\TL_n$.

\begin{theorem}\label{conethingthm}
For each sequence $\e \in \aL_n$ and convolution $[Q_\e]\in \inp{Q_\e}$,  
there exists a convolution $[Q_{\e\cdot (-1)}]\in\inp{Q_{\e\cdot (-1)}}$ and a chain map 
$$\d : [Q_\e]\sqcup 1 \rightarrow [Q_{\e\cdot (-1)}]$$ 
of homological and internal degree zero such that 
$$\Cone(\d)\in\inp{Q_{\e\cdot (+1)}}.$$
\end{theorem}

The proof below is a generalization of the obstruction theoretic argument
used to construct the map $\eta : P_{(1,1,-1)} \sqcup 1 \to t
P_{(1,1,-1,-1)}$ in Section \ref{fourstrandssec}. The convolution
$[Q_{\e\cdot (+1)}]$ will be defined as $\Cone(\d)$.

\begin{proof}
Let $S\subset \inp{Q_{\e}}$ denote the collection of chain complexes for
which the theorem is true.  In order to prove the theorem we show that
$Q_\e\in S$ and that $S$ is closed under convolution. These two statements
imply that $\inp{Q_\e}\subset S$.

In order to show that $Q_\e\in S$ we must chase our own definitions. By
definition, there is a chain complex $A_\e\in\Kom^*(n)$ such that $Q_\e = A_\e \otimes
P_k \otimes \overline{A}_\e$ where $|\e| = k$. By Lemma \ref{standardformlemma}, there is a triangle
$P_{k+1} = P_k\sqcup 1\buildrel \d'\over\longrightarrow t T$ where $T$
denotes the tail of the projector $P_{k+1}$.  Setting $[Q_{\e\cdot (-1)}] = (A_\e \sqcup
1)\otimes T\otimes (\overline{A}_\e\sqcup 1)$ and $\d =\Id\otimes \d'
\otimes \Id$ shows that $Q_\e\in S$.

The remainder of the proof shows that $S$ is closed under convolutions. Suppose that $[Q_\e]\in \inp{Q_\e}$ is a convolution. So $[Q_\e] = \Tot(E)$ where $E = \{ (E_i), q_{ij} \}$ and the
 $E_i\in \inp{Q_\e}$ are chain complexes for which the theorem holds.
By assumption there are chain complexes $T_i \in \inp{Q_{\e\cdot (-1)}}$ and maps
$$\d_i:E_i \sqcup 1\rightarrow T_i \conj{such that} \Cone(\d_i)\in \inp{Q_{\e\cdot (+1)}}.$$ 
We wish to define a chain complex $[Q_{\e\cdot (+1)}]\in\inp{Q_{\e\cdot (+1)}}$ which,
as a graded object, is a sum of the complexes appearing in the diagram below.
\begin{center}
\begin{tikzpicture}[scale=7, node distance=2.5cm]
\node (A) {$E_0 \sqcup 1$};
\node (A1) [right of=A] {$t E_1 \sqcup 1$};
\node (A2) [right of=A1] {$t^2 E_2 \sqcup 1$};
\node (B) [right of=A2] {$\cdots$};
\node (D) [below of=A] {$t T_0$};
\node (E) [right of=D] {$t^2 T_1$};
\node (F) [right of=E] {$t^3 T_2 $};
\node (G) [right of=F] {$\cdots$};
\draw[->] (A) to node {} (A1);
\draw[->] (A1) to node {} (A2);
\draw[->] (A2) to node {} (B);
\draw[->] (F) to node {} (G);
\draw[->] (A) to node {$\d_0$} (D);
\draw[->] (A1) to node {$\d_1$} (E);
\draw[->] (A2) to node {$\d_2$} (F);
\draw[->] (D) to node {} (E);
\draw[->] (E) to node {} (F);
\draw[->] (A) to node {} (E);
\draw[->] (A) to node {} (F);
\draw[->] (A1) to node {} (F);
\end{tikzpicture}
\end{center}

The convolution $[Q_{\e\cdot (+1)}]$ will be defined as a direct limit of
truncations $[Q_{\e\cdot (+1)}]_{[0,r]}$, see Definition \ref{truncdef}. We
proceed by induction on $r$.

If $r=0$ then set $[Q_{\e\cdot (+1)}]_{[0,0]} = \Cone(\d_0)$.
Assume that $[Q_{\e\cdot (+1)}]_{[0,r]}$ has constructed and let
$[Q_{\e}\sqcup 1]_{[0,r]}$ denote the corresponding truncation of
$[Q_\e]\sqcup 1$, which appears as the top row of $[Q_{\e\cdot (+1)}]_{[0,r]}$.  The
differential on $[Q_\e]\sqcup 1$ gives a chain map $\a:[Q_{\e}\sqcup
  1]_{[0,r]}\rightarrow t^rE_{r+1}\sqcup 1$.  Let $z$ be the
composition of the maps in the diagram below.

\begin{center}
\begin{tikzpicture}[scale=7, node distance=4cm]
\node (A) {$[Q_{\e\cdot (+1)}]$};
\node (A1) [right of=A] {$[Q_\e \sqcup 1]_{[0,r]}$};
\node (A2) [right of=A1] {$t^r E_{r+1}$};
\node (A3) [right of=A2] {$t^r T_{r+1}$};
\draw[->] (A) to node {$\pi$} (A1);
\draw[->] (A1) to node {$\alpha$} (A2);
\draw[->] (A2) to node {$(-1)^r \d_{r+1},$} (A3);
\end{tikzpicture}
\end{center}

\noindent where $\pi$ is the projection of $[Q_{\e\cdot (+1)}]_{[0,r]}$ onto its top row. By Theorem \ref{homComputation} the map $z$ belongs to a contractible $\Hom$-space:
$$z\in\Hom^*([Q_{\e \cdot (+1)}], t^r T_{r+1}) \simeq 0.$$
Therefore, $z = d(h)$ is a
boundary and Proposition \ref{obstprop} allows us to produce a chain complex
with $r+1$ terms of the desired form. 
\begin{center}
\begin{tikzpicture}[scale=3, node distance=3.5cm]
\node (A) {$[Q_{\e\cdot (+1)}]_{[0,r]}$};
\node (B) [right of=A] {$t^{r+1}E_{i+1}\sqcup 1$};
\node (D) [below of=A] {};
\node (E) [below of=B] {$t^{r+2} T_{r+1}$};
\draw[->] (A) to node {$\alpha \circ \pi$} (B);
\draw[->] (B) to node {$(-1)^{r+1} \delta_{r+1}$} (E);
\draw[->, bend right=25] (A) to node {$h$} (E);
\end{tikzpicture}
\end{center}

Observe that the differentials only point south or east.  The new column in
this complex is $t^{r+1}\Cone(\d_{r+1})$.  Since $\pi$ is the identity on
the top row and zero on the bottom row, the top row of this complex is the
corresponding truncation of $[Q_\e]\sqcup 1$. This construction is stable
for the same reasons as in Section \ref{fourstrandssec}.

Defining $[Q_{\e\cdot (+1)}]$ to be the limit of the resulting directed
system completes the proof.
\end{proof}

\begin{corollary}\label{pespexistenceprop}
For each sequence $\e\in \aL_n$, there exist a chain complex $P_\e \in \inp{Q_\e}$ in the hull of $Q_\e$ and maps
$$\delta_\e : P_\e \sqcup 1 \to t P_{\e \cdot (-1)}$$
such that $P_{\e \cdot (+1)} = \Cone(\delta_\e)$.
\end{corollary}

The corollary below follows from the argument given above.

\begin{corollary}
For each $n>0$ and each sequence $\e\in\aL_n$ the triangle
$$P_\e \sqcup 1 = P_{\e\cdot (-1)}\rightarrow P_{\e\cdot (+1)}$$
descends to Equation \eqref{peeqn} of Section \ref{genjwprojsec} in the
Grothendieck group $\K_0(\Kom(n))$.
\end{corollary}

The following theorem is a generalization of the resolution of identity
found in Section \ref{fourstrandssec}. This is the main result of this section.

\begin{theorem}\label{pespexistencethm}
For each $n>0$, there is a twisted complex $R_n = \{(P_\e), d_{\e\nu}\}_{\e\in\aL_n} \in \Tw\Kom(n)$ such that 
$$1_n \simeq R_n.$$
\end{theorem}
\begin{proof}
When $n=1$ we set $R_1 = 1$. Assume by induction that there is a twisted complex $R_{n-1} =\{ (P_\e), d_{\e\nu}\}_{\e\in\aL_{n-1}}$ such that $1_{n-1} \simeq R_{n-1}$. Placing a disjoint strand next to everything yields
$$1_n = 1_{n-1}\sqcup 1 =  R_{n-1}\sqcup 1 = \{ (P_\e\sqcup 1), d_{\e\nu}\sqcup 1\}.$$
Corollary \ref{substitution} and Corollary \ref{pespexistenceprop} imply that we can replace each $P_\e\sqcup 1$ with $P_{\e\cdot (-1)}\rightarrow P_{\e\cdot (+1)}$ obtaining an equivalence
\[
1_n\simeq \Big\{P_{\e\cdot (-1)}\xrightarrow{f_\e} P_{\e\cdot (+1)}, \left(\begin{smallmatrix} d_{P_{\e\cdot (-1)}} & 0\\ f_\e & d_{P_{\e\cdot (+1)}}\end{smallmatrix}\right)\Big\} \cong \{ P_{\e \cdot (\pm 1)}, d'_{\e\nu} \}
\]
The right-hand side is a twisted complex indexed by elements of
$\aL_n$. There may be maps which do not respect the dominance order $\dominatedBy$ on
$\aL_n$. However, when $\e \ndominatedBy \nu$ the $\Hom$-space from $P_\e$ to $P_\nu$ is contractible: 
$$\Hom^*(P_{\e},P_{\nu}) \simeq 0.$$

Applications of the Combing Lemma \ref{combinTheHairs} allow us to exchange
maps in $\{ P_{\e \cdot (\pm 1)}, d'_{\e\nu} \}$ which do not respect the
dominance order for those that do. The resulting twisted complex is the
resolution of identity $R_n$.
\end{proof}

\begin{remark}
  When referring to a chain complex in $R_n \in \Kom(n)$ the resolution of
  identity $R_n$ is defined to be the convolution $\Tot(R_n)$, see
  Definition \ref{totdef}.
\end{remark}

\begin{remark}
  In the Grothendieck group $\K_0(\Kom(n))$, the resolution of identity
  becomes the equation $1_n = \sum_{\e\in\aL_n} p_\e $ from Section
  \ref{genjwprojsec}. From the discussion in Section \ref{idempotentsirred sec}, we see that $R_n$ categorifies the decomposition of $V_1^{\otimes
    n}$ into irreducible representations.
\end{remark}

\begin{remark}
In the decategorified setting, representations decompose into direct sums of
irreducible representations. After categorification we have learned that
this decomposition is maintained up to homotopy, but the irreducible
components now have non-trivial maps between them.
\end{remark}

\section{Higher order projectors}\label{charprojsec}
In this section we will define the universal higher order projectors and
articulate the sense in which the resolutions of identity $R_n$ produced in
Sections \ref{explicitsec} and \ref{mainconstructionsec} yield
categorifications of the idempotents $P_{n,k}$ defined in Section
\ref{idempotentsirred sec}. While the axioms of Definition
\ref{ukgenprojdef} given below are sufficient to characterize the projectors
$P_{n,k}$ uniquely up to homotopy we will see that the $P_{n,k}$ also
satisfy a number of other useful properties analogous to those enumerated in
Section \ref{genjwprojsec}. We begin by introducing a few definitions
similar to those of Section \ref{TLcat}.

Just as elements $a\in \TL(n,m)$ have a notion of through-degree (Definition
\ref{throughdegdef}), chain complexes $A\in\Kom(n,m)$ have a corresponding
notion of through-degree.

\begin{definition}\label{komthroughdegdef}
Suppose that $A \in \Cob(n,m)$ is a Temperley-Lieb diagram then $A$ factors
as a composition $A = C\otimes B$ where
$$B\times C \in \Cob(n,l)\times \Cob(l,m).$$
The \emph{through-degree} $\tau(A)$ of $A$ is equal to the minimal $l$ in
such a factorization. If $A\in \Kom(n,m)$ is chain complex of
Temperley-Lieb diagrams $\{A_i\}$ then $\tau(A)=\max_i \tau(A_i)$.
\end{definition}

We now define subcategories $\Kom^k(n)$ of $\Kom(n)$. Elements of
$\Kom^k(n)$ will be convolutions of complexes which factor through the
universal projector $P_k$. If $A\in\Kom^k(n)$ then $\tau(A) = k$. Compare to
Remark \ref{REMREF}.

\begin{definition}\label{subcatkom}
A chain complex $C\in\Kom(n,m)$ \emph{factors through} $P_k$ when there are objects $A\in \Kom(n,k)$ and $B \in \Kom(k,m)$ such that
$$C \cong A \otimes P_k \otimes B.$$
Let $\Kom^k(n,m) \subset\Kom(n,m)$ be the full subcategory of convolutions
of chain complexes which factor through $P_k$. We have analogous notions of
subcategories $\Kom^{*,k}(n,m)$ and $\Tw\Kom^k(n,m)$ in $\Kom^*(n,m)$ and
$\Tw\Kom(n,m)$ respectively. See Section \ref{twistedcomsec} for definitions
of these categories.
\end{definition}

The next lemma tells us that chain complexes which factor through various
universal projectors $P_k$ compose in a predictable manner.

\begin{lemma}\label{ortholemma}
For each $A\in \Kom^k(n,m)$ and $B\in \Kom^l(m,r)$, if $k \neq l$ then 
$$B\otimes A\simeq 0.$$  
\end{lemma}

\begin{proof}
  Observe that composing complexes which factor through projectors $P_k$ and
  $P_l$ with $k \neq l$ produces a complex containing a turnback; this is
  contractible by Theorem \ref{uprojectorthm}. The composite twisted complex
  lies in the hull of a collection of contractible complexes and therefore
  it is contractible by Lemma \ref{contractLemma}.
\end{proof}

We will now state what is meant by universal higher order projectors. 

\begin{definition}\label{ukgenprojdef}
A chain complex $P \in \Kom(n)$ is a $k$th \emph{universal higher order 
  projector} if
\begin{enumerate}
\item The through-degree $\tau(P)$ of $P$ is equal to $k$.

\item $P$ vanishes when the number of turnbacks is sufficiently high. For each $l\in\ZZ_+$ and $a \in \Cob(n,l)$ if $\tau(a) < k$ then
$$a \otimes P \simeq 0 \quad\normaltext{ and }\quad P \otimes \bar{a} \simeq 0.$$ 

\item There exists a chain complex $C \in \Kom(n)$ with $\tau(C) <  k$ and a twisted complex 
$$D =  1_n\rightarrow C\rightarrow t P $$ 
such that 
$$a\otimes D \simeq 0\quad \normaltext{ and }\quad D\otimes \bar{a} \simeq 0$$
for all diagrams $a\in \Cob(n,m)$ such that $\tau(a)\leq k$.
\end{enumerate}
\end{definition}

For each sequence $\e\in \aL_n$, there is a complex $Q_\e$ (see Definition
\ref{bigqedef}) and if $|\e| = k$ then $Q_\e$
factors through $P_k$ by construction. It follows that each object $A
\in\inp{Q_\e}$ must factor through $P_k$. In particular, $P_\e \in
\Kom^{|\e|}(n)$ for each $\e\in\aL_n$.

As in Definition \ref{generalizedprojdef}, the
constitutents of the higher order projector $P_{n,k}$ consist of projectors
$P_\e$ with $\e\in \aL_{n,k}$. The categorical construction differs in that
there are now non-trivial maps between the components, $P_{\e}$. We extract
$P_{n,k}$ from the resolution of identity in the definition below.

\begin{definition}\label{ugenprojdef}
A $k$th \emph{universal higher order projector} $P_{n,k}$ is the convolution
of the subcomplex formed by isotypic components in the resolution of
identity.
$$P_{n,k} = \left(\displaystyle\bigoplus_{\e\in\aL_{n,k}} P_\e, d_{\e\nu}\right)$$
\end{definition}

\begin{remark}
  The projector $P_{n,k}$ is homotopy equivalent to the cone of the inclusion map between relative truncations of $R_n$, see Equation \eqref{coneeqn}.
\end{remark}

\begin{theorem}\label{pnkarepnkthm}
The chain complexes $P_{n,k}$ of Definition \ref{ugenprojdef} are universal higher order projectors.
\end{theorem}

Before proving the theorem we record several observations. By construction $P_{n,k}$ is contained in the hull of the set $\aQ = \{ Q_\e : \e \in \aL_{n,k} \}$. The next observation follows from the discussion preceding the definition.

\begin{observation}\label{obs4}
The $k$th universal higher order projector factors through the universal projector $P_k$. In particular,  $P_{n,k} \in \Kom^k(n).$
\end{observation}

Since each $P_{n,k}$ is the restriction of the resolution of identity to the
subcomplex consisting of the $P_{\e}$ with $|\e| = k$ we can write a
resolution of identity purely in terms of the higher order projectors.

\begin{observation}\label{obs6}
$$1_n \simeq P_{n,n \imod{2}} \rightarrow \cdots \rightarrow P_{n,n-4} \rightarrow P_{n,n-2}\rightarrow P_{n,n}$$
The diagram on the right usually contains higher differentials, $P_{n,i} \to P_{n,j}$, when $i<j$.
\end{observation}

We are ready to prove that the chain complexes $P_{n,k}$ extracted from the
resolution of identity above satisfy the properties listed in Definition
\ref{ukgenprojdef}.

\begin{proof}{(of Theorem \ref{pnkarepnkthm})}
The first property, $\tau(P_{n,k}) = k$, follows from
Observation \ref{obs4} above.

Suppose that $a \in \Cob(n,l)$ is a diagram with $\tau(a) < k$. Again, by
Observation \ref{obs4} the complex $a\otimes P_{n,k}$ is contained in the
hull of $\{a\otimes Q_\e : |\e| = k\}$, but $a\otimes Q_\e \simeq 0$ because
$Q_\e$ factors through $P_k$. Each complex in the hull of a collection of
contractible complexes is contractible by Lemma \ref{contractLemma}.

Rotating distinguished triangles in Observation \ref{obs6} above
gives the homotopy equivalence
$$ 1_n\rightarrow t(P_{n, n \imod{2}}\rightarrow \cdots \rightarrow P_{n,k-2}) \rightarrow P_{n,k} \simeq P_{n,k+2}\rightarrow \cdots\rightarrow P_{n,n}$$
Let $D$ be the left-hand side of this equation and set $C$ to be the middle term
$$C = t(P_{n, n\imod{2}}\rightarrow \cdots \rightarrow  P_{n,k-2})$$
so that the third property follows.
\end{proof}

We have seen that the chain complexes $P_{n,k}$ defined above are $k$th
universal higher order projectors. The next theorem states that each chain
complex $P$ satisfying the properties of Definition \ref{ukgenprojdef} is
homotopy equivalent to the chain complex $P_{n,k}$.

\begin{theorem}\label{pnkareuniquethm}
If $P\in\Kom(n)$ is a $k$th universal higher order projector then $P$ is homotopy equivalent to $P_{n,k}$ of Definition \ref{ukgenprojdef}.
\end{theorem}
\begin{proof}
Suppose that $P \in \Kom(n)$ satisfies properties (1)--(3) of Definition \ref{ukgenprojdef} above.  From Observation \ref{obs6} we have the resolution of identity.
$$1_n \simeq P_{n,n \imod{2}} \rightarrow \cdots \rightarrow P_{n,n-4} \rightarrow P_{n,n-2}\rightarrow P_{n,n}$$
Applying $P\otimes -$ above yields $P = P\otimes 1_n \simeq P_{n,k}\otimes P$.
By Property (3) there are complexes $C$ and $D$ where
$$D = 1_n\rightarrow C\rightarrow t P.$$
Since $\tau(P_{n,k})=k$, Property (3) also implies that $P_{n,k}\otimes D\simeq 0$. Now $P_{n,k}\otimes D\simeq 0$ tells us that
\[
0 \simeq P_{n,k}\otimes 1_n \rightarrow P_{n,k} \otimes C \rightarrow tP_{n,k}\otimes P = \Cone(P_{n,k} \rightarrow P_{n,k}\otimes P)
\]
because $P_{n,k} \otimes C \simeq 0$. However, since $\Cone(f)\simeq 0$ if and only if
$f$ is a homotopy equivalence, the above equation implies
that $P_{n,k}\otimes P\simeq P$ and therefore $P_{n,k}\simeq P$.
\end{proof}

Now that existence and uniqueness have been shown, we continue our
discussion with a series of observations.

\begin{proposition}
The top projector $P_{n,n}$ is a universal projector $P_n$. 
\end{proposition}

\begin{proof}
  This can be seen indirectly by comparing the three properties found in
  Definition \ref{ukgenprojdef} with those of Theorem \ref{uprojectorthm}.
  Alternatively, this can be seen directly by tracing through the
  construction in either Section \ref{explicitsec} or Section
  \ref{mainconstructionsec}. See \cite{CK} for an extended discussion of the
  axioms found in Theorem \ref{uprojectorthm}.
\end{proof}

\begin{remark}
The universal projectors $P_{n,n}$ of Definition \ref{uprojectorthm} were
first categorified in \cite{CK, R1, FSS}.  The bottom projectors
$P_{2n,0}\in\Kom(n)$ were categorified and related to the Hochschild
homology of Khovanov's ring $H_n$ in \cite{R2}.
\end{remark}

\begin{proposition}
The $P_{n,k}$ are mutually orthogonal idempotents,
$$P_{n,k}\otimes P_{n,l}\simeq \delta_{k,l} P_{n,k}.$$
\end{proposition}


\begin{proof}
When $k\neq l$ the statement $P_{n,k} \otimes P_{n,l} \simeq 0$ follows from
Observation \ref{obs4} and Lemma \ref{ortholemma} above. If $k=l$ then
consider the resolution of identity
$$1_n \simeq P_{n,n\imod{2}} \rightarrow \cdots \rightarrow P_{n,n-2}\rightarrow P_{n,n}$$ 
composing with $P_{n,k}$ gives $P_{n,k} = P_{n,k}\otimes 1_n \simeq P_{n,k}\otimes P_{n,k}$.
\end{proof}

The proof of the proposition below is analogous to the proof of Proposition
\ref{slidethroughobs}.

\begin{proposition}
Suppose that $a\in\Kom(m,n)$ then
$$a\otimes P_{m,k} \simeq P_{n,k}\otimes a.$$
In pictures,
$$\CPic{atopnk} \simeq \hspace{.5in}\CPic{mktopa}.$$
\end{proposition}

Theorem \ref{homComputation} implies that
$\Hom$-spaces between convolutions of $Q_\e$ and $Q_\nu$ are contractible
when $\e \ndominatedBy \nu$. Since the complexes $P_\e$ are convolutions of
$Q_\e$ and $P_{n,k}$ is constructed using $P_\e$ with $\e\in\aL_{n,k}$,
$$\Hom^*(P_{n,i},P_{n,j})\simeq 0 \conj{when} j < i.$$






Theorem \ref{modulethmdef} implies that the differential graded algebra $E =
\oplus_{i\leq j} \Hom^*(P_{n,i},P_{n,j})$ is of fundamental importance. This
algebra is a generalization of the endomorphism algebras $\Endo^*(P_n)$. The
precise nature of the algebra $\Endo^*(P_n)$ has been the focus of a series
of conjectures by the authors Gorsky, Oblomkov, Rasmussen and Shende, see
\cite{GOR1}.

\subsection{Postnikov decompositions}\label{postnikovdecompsec}

In this section we discuss a more algebraic characterization of the
projectors $P_{n,k}$ and the resolution of identity $R_n$.

\begin{definition}
The resolution of identity $R_n$ can be written as a convolution of higher order projectors:
$$1_n \simeq R_n = P_{n,n \imod{2}} \rightarrow \cdots \rightarrow P_{n,n-4} \rightarrow P_{n,n-2}\rightarrow P_{n,n},$$
see Observation \ref{obs6}. Truncating the diagram yields triangles of {\em inhomogeneous idempotents}.
$$R_n = W_{n,k} \to Z_{n,k}$$
where
$$W_{n,k} = P_{n,n\imod{n}} \to \cdots \to P_{n,k-2} \quad\normaltext{ and }\quad Z_{n,k} = P_{n,k} \to P_{n,k+2} \to \cdots \to P_{n,n}.$$
For each $k$, there is a canonical inclusion $Z_{n,k} \to Z_{n,k-2}$, which yields a triangle 
\begin{equation}\label{coneeqn}
P_{n,k-2} \to Z_{n,k} \to Z_{n,k-2}
\end{equation}
\end{definition}

It follows from the triangles in Equation \eqref{coneeqn} that the
inhomogeneous idempotents $Z_{n,k}$ determine the higher order projectors
$P_{n,k}$ up to homotopy. We will show that the $Z_{n,k}$ can be
characterized by localization.

\begin{remark}
For any object $S$, these triangles can be sewn together to obtain the canonical decomposition pictured below.
\begin{center}
\begin{tikzpicture}[scale=10, node distance=2.5cm]
\node (A) {$\cdots$};
\node (B) [right=1cm of A] {$S\otimes Z_k$};
\node (BT) [above=1cm of B] {$S\otimes P_{n,k-2}$};
\node (C) [right of=B] {$S\otimes Z_{k-2}$};
\node (CT) [above=1cm of C] {$S\otimes P_{n,k-4}$};
\node (D) [right=1cm of C] {$\cdots$};
\node (E) [right=1cm of D] {$S \otimes R_n \cong S$};
\draw[->] (A) to node {} (B);
\draw[->] (BT) to node {} (B);
\draw[->] (B) to node {} (C);
\draw[->] (C) to node {} (D);
\draw[->] (CT) to node {} (C);
\draw[->] (D) to node {} (E);
\end{tikzpicture}
\end{center}
\end{remark}

The picture above is a Postnikov type decomposition. It is well known that
such decompositions can be constructed functorially using the language of
quotients and localization.  In Proposition \ref{bousprop} we will show that
the chain complex $Z_{n,k}$ determines the Bousfield colocalization with
respect to the $k$th layer of the through-degree filtration.

We must extend the category $\Kom(n)$ so that certain limits are guaranteed
to exist.

\begin{definition}{($\Ho_{+}(n)$)}\label{komoplusdef}
  Let $\Kom_{+}(n,m) = \Kom_+(\Cob(n,m))$ be the closure of $\Kom(n,m)$
  under small coproducts.  We will use the abbreviations:
  $$\Kom_{+}(n) = \Kom_{+}(n,n)\conj{ and } \Ho_+(n) = \Ho(\Kom_+(n)).$$
\end{definition}

The next definition should be compared to Remark \ref{REMREF}.

\begin{definition}
  The category $\Ho_+(n)$ is filtered by through-degree $\tau$, see
  Definition \ref{komthroughdegdef}.  If $\Kom^k_+(n)\subset \Kom_+(n)$ is
  the full subcategory consisting of chain complexes of diagrams with
  through-degree less than $k$ then by setting
$$\Ho^k_+(n) = \Kom^k_+(n)$$
we obtain the filtration
$$\cdots \subset \Ho_+^{k-1}(n) \subset \Ho_+^{k}(n) \subset \Ho^{k+1}_+(n) \subset \cdots\conj{ and } \Ho_+(n) = \cup_k \Ho_+^k(n).$$
\end{definition}

The resolution of identity $R_n$ and the projectors $P_{n,k}$ exist in
$\Kom_{+}(n)$ and the inhomogeneous projector $Z_{n,k}$ induces a functor:
$$-\otimes Z_{n,k} : \Ho_+(n)\to \Ho_+(n).$$ 
We would like to show that this functor is the universal functor which
annihilates subcategory $\Ho^k_+(n)$. Before stating the proposition we
recall some vocabulary.

\begin{definition}
Let $\aT$ be a triangulated category. Then $L : \aT\to \aT$ is a {\em localization functor} when $L$ is exact and there is a map $\eta : \Id_{\aT}\to L$ such that $L\eta : L \to L^2$ is invertible and $L\eta = \eta L$. A functor $L: \aT\to \aT$ is a {\em colocalization functor} when $L^{\op} : \aT^{\op} \to \aT^{\op}$ is a localization functor.
\end{definition}

\begin{definition}
Let $L : \aT \to \aT$ be a colocalization functor. Then an object $X\in \aT$ is {\em $L$-acyclic} when $LX \cong 0$ and the {\em kernel of $L$} $\ker(L)\subset \aT$ is the subcategory formed by all $L$-acyclic objects.
\end{definition}

For further detail the reader may consult \cite[Ch. 9]{Neeman}.

\begin{proposition}\label{bousprop}
The functor determined by the inhomogeneous projector
$$-\otimes Z_{n,k} : \Ho^+(n)\to \Ho^+(n)$$ 
is a colocalization functor and $\ker(-\otimes Z_{n,k}) = \Ho^k_+(n)$.
\end{proposition}
\begin{proof}
The functor $-\otimes Z_{n,k}$ is exact because it commutes with cones and suspension. 

If $S\in \Ho^k_+(n)$ then $S$ is a chain complex consisting of diagrams with through-degree less than $k$ then the second property of Definition \ref{ukgenprojdef} implies that $S\ott Z_{n,k} \cong 0$. Conversely, if $S\in\Ho_+(n)$ and $S\ott Z_{n,k} \cong 0$ then 
$$S \cong S\ott Z_{n,k} \to S\ott W_{n,k} \cong S\ott W_{n,k}$$
implies that $\tau(S)=\tau(S\ott W_{n,k})\leq \tau(W_{n,k})< k$ and therefore $S\in \Ho_+^k(n)$.  We conclude that $\ker(-\otimes Z_{n,k}) = \Ho^k_+(n)$.

There is a canonical map $j_{n,k} : Z_{n,k} \to R_n$. If $S\in \Ho_+(n)$
then define $\eta_{S} : S\ott Z_{n,k} \to S$ by the composition $\varphi_S
\circ (\Id_S \otimes j_{n,k})$ where $\varphi_S : R_n \otimes S
\xrightarrow{\sim} S$. The naturality of the map $\Id_S \ott j_{n,k}$
follows from the definition of monoidal structure. The map $\varphi_S$ is
natural in $S$ because it can be written as $S\ott R_n = \Cone(S \to C_S)
\xrightarrow{\sim} S$ where $C_S$ is a contractible chain
complex. Therefore, $\eta$ is a natural transformation. The properties of
$\eta$ follow from idempotence, $Z_{n,k}\otimes Z_{n,k}\cong Z_{n,k}$, and
associativity relations.

\end{proof}

\noindent {\bf Acknowledgements.}  The authors would like thank V. Krushkal for his
interest. 

\noindent B. Cooper was supported in part by the Max Planck Institute for Mathematics in
Bonn.

\end{document}